\documentclass[11pt,leqno]{amsart}
\usepackage{amssymb}
\usepackage{mathrsfs}
\usepackage{color}
\usepackage{soul}
\usepackage{dsfont}

\newcommand\N{{\mathbb N}}
\newcommand\R{{\mathbb R}}

\newcommand\C{{\mathbb C}}

\newcommand\EEE{{\bf E}}

\def\AA{{\mathcal A}}
\def\BB{{\mathcal B}}
\def\CC{{\mathcal C}}
\def\DD{{\mathcal D}}
\def\EE{{\mathcal E}}
\def\FF{{\mathcal F}}

\def\JJ{{\mathcal J}}

\def\LL{{\mathcal L}}

\def\NN{{\mathcal N}}
\def\OO{{\mathcal O}}

\def\RR{{\mathcal R}}

\def\TT{{\mathcal T}}
\def\UU{{\mathcal U}}

\def\BBB{{\mathscr{B}}}
\def\LLL{{\mathscr{L}}}

\def\eps{{\varepsilon}}

\newtheorem{theo}{Theorem}
\newtheorem{prop}[theo]{Proposition}
\newtheorem{lem}[theo]{Lemma}
\newtheorem{cor}[theo]{Corollary}

\newcommand{\beqn}{\begin{equation}}
\newcommand{\eeqn}{\end{equation}}
\newcommand{\bear}{\begin{eqnarray}}
\newcommand{\eear}{\end{eqnarray}}
\newcommand{\bean}{\begin{eqnarray*}}
\newcommand{\eean}{\end{eqnarray*}}

\def\Nt{|\hskip-0.04cm|\hskip-0.04cm|}

\newcommand{\Black}{\color{black}}

\begin{document}

\title[Discrete, fractional \& classical Fokker-Planck equations]{Uniform semigroup spectral analysis of the discrete, fractional  \& classical Fokker-Planck equations}

\date\today
 
\author[S. Mischler]{St\'{e}phane Mischler}
\address[S. Mischler]{Universit\'e Paris Dauphine \& IUF, Ceremade (UMR 7534), Place du Mar\'echal de Lattre de Tassigny, 75775 Paris Cedex 16 France.}
\email{mischler@ceremade.dauphine.fr}

\author[I. Tristani]{Isabelle Tristani}
\address[I. Tristani]{\'{E}cole Polytechnique, Centre de MathÃÂ©matiques Laurent Schwartz UMR 7640, 91128 Palaiseau Cedex,
France.}
\email{isabelle.tristani@polytechnique.edu}

\begin{abstract}
In this paper, we investigate the spectral analysis and long time asymptotic convergence of semigroups associated to discrete, fractional and classical Fokker-Planck equations in some regime where the corresponding operators are close.  
We successively deal with the discrete and the classical Fokker-Planck model, the fractional and the classical Fokker-Planck model and  finally the fractional and the classical Fokker-Planck model. In each case, we prove uniform spectral estimates using perturbation and/or enlargement arguments.
%
\end{abstract}

\maketitle



 \smallskip


%
%
%
%
%
	
\vspace{0.3cm}



\bigskip


\tableofcontents



\section{Introduction} 
\label{sec:intro}
\setcounter{equation}{0}
\setcounter{theo}{0}
 
 \subsection{Models and main result}
In this paper, we are interested in the spectral analysis and the long time asymptotic convergence of semigroups associated to some discrete, fractional and classical Fokker-Planck equations.
They are simple models for describing the time evolution of a density function $f = f(t,x)$, $t \ge 0$, $x \in \R^d$, of particles 
undergoing both diffusion and (harmonic) confinement mechanisms and write
\beqn\label{eq-intro:dtf=Lambdaf}
 \partial_t f = \Lambda f = \DD f  + \hbox{div}(xf), \quad f(0) = f_0.
\eeqn
The diffusion term may either be a discrete diffusion 
 $$
\DD f = \Delta_\kappa  f :=  \kappa * f -  \|Â \kappa \|_{L^1} f, 
$$
 for a convenient (at least nonnegative and symmetric) kernel $\kappa$. It can also be a fractional diffusion  
\bear \label{eq-intro:opfrac}
(\DD f) (x) &=&  -(-\Delta)^{\alpha \over 2} f(x) 
\\ \nonumber
&:= &
c_\alpha \int_{\R^d} {f(y) - f(x) -\chi(x-y) (x-y)  \cdot \nabla f (x) \over |x-y|^{d+\alpha}} \, dy, 
\eear
with  $\alpha \in (0,2)$, $\chi \in \DD(\R^d)$ radially symmetric satisfying the inequality $\mathds{1}_{B(0,1)}Â \le \chi \le \mathds{1}_{B(0,2)}$, and a convenient normalization constant $c_\alpha > 0$. It can finally be the classical diffusion
$$
\DD f = \Delta f := \sum_{i=1}^d \partial^2_{x_ix_i} f.
$$
The main features of these equations are (expected to be) the same: they are mass preserving, namely 
$$
\langle f(t) \rangle = \langle f_0 \rangle, \quad \forall \, t \ge 0,  \quad \langle f \rangle := \int_{R^d} f \, dx, 
$$
positivity preserving, have
a unique positive stationary state with unit mass and that  stationary state is exponentially stable, in particular
\beqn\label{eq-intro:fTO0atINFINITY}
f(t) \to 0 \quad\hbox{as}\quad t\to\infty, 
\eeqn
for any solution associated to an initial datum $f_0$ with vanishing mass.  Such results can be obtained using different tools
as the spectral analysis of self-adjoint operators, some (generalization of) Poincar\'e inequalities or logarithmic Sobolev inequalities 
as well as the Krein-Rutman theory for positive semigroup.  
 
  \smallskip
 The aim of this paper is to initiate a kind of unified treatment of the above generalized Fokker-Planck equations and more importantly to establish that the convergence
 \eqref{eq-intro:fTO0atINFINITY} is exponentially fast uniformly with respect to the diffusion term for a large class of initial data which are taken in a 
 fixed weighted Lebesgue or weighted Sobolev space $X$.  
 
 \smallskip
 We investigate three regimes where these diffusion operators are close and for which 
 such a uniform convergence can be established.  In Section \ref{sec:DFP-FP}, we first consider  the case when the diffusion operator is discrete 
$$
\DD f  = \DD_\eps f := \Delta_{\kappa_\eps} f, \quad \kappa_\eps :=  {1 \over \eps^2}  \, k_\eps , 
$$
where $k$ is a nonnegative, symmetric, normalized, smooth and decaying fast enough  kernel and where we use the notation $k_\eps (x) = k(x/\eps)/\eps^d$, $\eps>0$. 
In the limit $\eps \to 0$, one then recovers the classical diffusion operator $\DD_0 = \Delta$. 
 
\smallskip
In Section \ref{sec:FFP-FP}, we next consider the case when the diffusion operator is fractional 
$$
\DD f  = \DD_\eps f  := - (-\Delta)^{(2-\eps)/2} f, \quad \eps \in (0,2),
$$
so that in the limit $\eps \to 0$ we also recover the classical diffusion operator $\DD_0 = \Delta$. 

 \smallskip
In Section  \ref{sec:DFFP-FFP}, we finally consider the case when the diffusion operator is a discrete version of the fractional diffusion, namely 
$$
\DD f =  \DD_\eps f  := \Delta_{ \kappa_\eps} f  , 
$$
where $(\kappa_\eps)$ is a family of convenient bounded kernels which converges towards the kernel of the fractional diffusion operator $k_0 := c_\alpha \, |Â \cdot |^{-d-\alpha}$ for some fixed $\alpha \in (0,2)$, in particular, in the limit $\eps \to 0$, one may recover the fractional  diffusion operator $\DD_0 =-(- \Delta)^{\alpha/2}$. 
 
 In order to write a rough version of our main result, we introduce some notation. We define the weighted Lebesgue space $L^1_r$, $r \ge 0$,  as the space of measurable functions $f$ such that  $f \, \langle x \rangle^r \in L^1$, where $\langle x \rangle^2 := 1 + |x|^2$.  For any $f_0 \in L^1_r$, we denote as 
$f(t)$ the solution to the generalized Fokker-Planck equation \eqref{eq-intro:dtf=Lambdaf} with initial datum $f(0) = f_0$ and then we define the semigroup $S_{\Lambda}$ on $X$ by setting $S_\Lambda (t) f_0 := f(t)$. 
 
 \begin{theo}[rough version]\label{theo:RoughVersion} 
 There exist $q > 0$ and $\eps_0 \in (0,2)$ such that for any $\eps \in [0,\eps_0]$, the semigroup $S_{\Lambda_\eps}$ is well-defined on $X := L^1_r$ and there exists a unique positive and normalized stationary solution $G_\eps$ to  \eqref{eq-intro:dtf=Lambdaf}. Moreover, there exist $a < 0$ and $C \ge 1$ such that for any $f_0 \in X$, there holds 
\beqn\label{eq:theo:RoughVersion}
 \|S_{\Lambda_\eps} (t) f_0 -G_\eps \langle f_0 \rangle \|_X \le C \, e^{at} Â \, \|Â f_0 - G_\eps \langle f_0 \rangle \|_X, \quad \forall \, t \ge 0.
\eeqn
 \end{theo}
 
Our approach is a semigroup approach in the spirit of the semigroup decomposition framework introduced by Mouhot in \cite{Mcmp} and developed 
 subsequently in \cite{MMcmp,GMM,FFP-IT,MM*,Mbook*}.  Theorem~\ref{theo:RoughVersion} generalizes to the discrete diffusion Fokker-Planck equation and to the discrete fractional diffusion Fokker-Planck equation similar results obtained for the classical Fokker-Planck equation in \cite{GMM,MM*} (Section~\ref{sec:DFP-FP}) and for the fractional one in \cite{FFP-IT} (Section~\ref{sec:DFFP-FFP}). It also makes uniform with respect to the fractional diffusion parameter  the convergence results obtained for the fractional diffusion equation in \cite{FFP-IT} (Section \ref{sec:FFP-FP}). 
 It is worth mentioning that there exists a huge literature on the long-time behaviour for the Fokker-Planck equation as well as (to a lesser extend) for the fractional Fokker-Planck equation. We refer to the references quoted in \cite{GMM,MM*,FFP-IT} for details. There also probably exist many papers on the discrete diffusion equation since it is strongly related to a standard random walk in $\R^d$, but we were not able to find any precise reference in this PDE context.

 \smallskip
\subsection{Method of proof}
 Let us explain our approach. First, we may associate a semigroup $S_{\Lambda_\eps}$ to the evolution equation~\eqref{eq-intro:dtf=Lambdaf} in many Sobolev spaces, 
 and that semigroup is mass preserving and positive. In other words, $S_{\Lambda_\eps}$ is a Markov semigroup and it is then expected that there exists a unique positive and unit mass steady state 
$G_\eps$ to the equation \eqref{eq-intro:dtf=Lambdaf}. Next, we are able to establish that the semigroup $S_{\Lambda_\eps}$ splits as 
\bear\label{eq-intro:SLambdaSplitting}
&& S_{\Lambda_\eps}  Â =  S^1_{\eps}  Â +  S^2_{\eps}, \quad
\\ \nonumber
&& S^1_\eps  Â  \approx e^{t T_\eps}, \,\, T_\eps \hbox{ finite dimensional}, 
\quad
S^2_\eps Â = \OO(e^{at}), \,\, a < 0,
\eear
in these many weighted Sobolev spaces. The above decomposition of the semigroup is the main technical issue of the paper. It is obtained by 
introducing a convenient splitting 
\beqn\label{eq-intro:Splitting}
\Lambda_\eps = \AA_\eps + \BB_\eps
\eeqn
where $\BB_\eps$ enjoys suitable dissipativity property and $\AA_\eps$ enjoys some suitable $\BB_\eps$-power regularity (a property that we introduce in Section~\ref{subsec:AisBpoweReg} (see also \cite{Mbook*}) and that we name in that way by analogy with the $\BB_\eps$-power compactness notion introduced by Voigt  \cite{Voigt80}). It is worth emphasizing that we are able to exhibit such a splitting with uniform (dissipativity, regularity) estimates
with respect to the diffusion parameter $\eps \in [0,\eps_0]$ in several weighted Sobolev spaces.  

As a consequence of \eqref{eq-intro:SLambdaSplitting}, we may indeed apply the Krein-Rutman theory developed in \cite{MS,Mbook*} and exhibit such a unique positive and unit mass steady state  $G_\eps$. 
Of course for the classical and fractional Fokker-Planck equations the steady state is trivially given through an explicit formula (the Krein-Rutman theory is useless in that cases). 
A next direct consequence of the above spectral and semigroup decomposition \eqref{eq-intro:SLambdaSplitting} is that 
there is a spectral gap in the spectral set $\Sigma(\Lambda_\eps)$ of the generator $\Lambda_\eps$, namely 
\beqn\label{eq-intro:SpectralGap}
\lambda_\eps := \sup \{Â \Re e \, \xi \in \Sigma(\Lambda_\eps) \backslash \{Â 0 \} \}Â < 0, 
\eeqn
and next that an exponential trend to the equilibrium can be established, namely 
\beqn\label{eq-intro:intro:SLtoG}
\|  S_{\Lambda_\eps}(t)  Â f_0  \|_X \le C_\eps \, e^{at} \, \|  f_0 \|_X \quad \forall \, t \ge 0, \,\,  Â \forall \, \eps \in [0,\eps_0], \,\, \forall \, a > \lambda_\eps, 
\eeqn
 for any  initial datum $f_0 \in X$ with vanishing mass. 
 
 \smallskip
 Our final step consists in proving that the spectral gap \eqref{eq-intro:SpectralGap} and the estimate \eqref{eq-intro:intro:SLtoG} are uniform with respect to $\eps$, more precisely, 
 there exists $\lambda^* < 0$ such that $\lambda_\eps \le \lambda^*$ for any $\eps \in [0,\eps_0]$ and $C_\eps$ can be chosen independent to $\eps \in [0,\eps_0]$. 

 \smallskip
 A first way to get such uniform bounds is just to have in at least one Hilbert space $E_\eps \subset L^1(\R^d)$ the estimate 
 $$
 \forall \, f \in \DD(\R^d), \,\, \langle f \rangle = 0, \quad 
( \Lambda_\eps f, f )_{E_\eps} \le \lambda^* \|Â f \|^2_{E_\eps}, 
$$
 and then \eqref{eq-intro:intro:SLtoG} essentially follows from the fact that the splitting \eqref{eq-intro:Splitting} holds with operators which are uniformly bounded with respect to $\eps \in [0,\eps_0]$. It is the strategy we use in the case of the fractional diffusion (Section~\ref{sec:FFP-FP}) and the work has already been made in \cite{FFP-IT} except for the simple but fundamental observation that the fractional diffusion operator is uniformly bounded (and converges to the classical diffusion operator) when it is suitable (re)scaled. 
 
 \smallskip
 A second way to get the desired uniform estimate is to use a perturbation argument. Observing that, in the discrete cases (Sections \ref{sec:DFP-FP} and \ref{sec:DFFP-FFP}),
 $$
 \forall \, \eps \in [0,\eps_0], \quad \Lambda_\eps - \Lambda_{0}Â = \OO (\eps),
 $$
for a suitable operator norm, we are able to deduce that $\eps \mapsto \lambda_\eps$ is a continuous function at $\eps=0$, from which we readily conclude. 
We use here again that the considered models converge to the classical or the fractional Fokker-Planck equation. 
In other words, the discrete models can be seen as (singular) perturbations of the limit equations and our analyze
takes advantage of such a property in order to capture the asymptotic behaviour of the related spectral objects (spectrum, spectral projector) and to conclude 
to the above uniform spectral decomposition. This kind of perturbative method has been introduced in \cite{MMcmp} and improved in \cite{Granular-IT*}. In Section \ref{sec:DFFP-FFP}, we give a new and improved version of the abstract perturbation argument where some dissipativity assumptions are relaxed with respect to \cite{Granular-IT*} and  only required to be satisfied on the limit operator ($\eps=0$).

\smallskip
\subsection{Comments and possible extensions}
~~\\
\noindent
{\sl Motivations.} The main motivation of the present work is rather theoretical and methodological. Spectral gap and semigroup estimates in large Lebesgue spaces
have been established both for Boltzmann like equations and Fokker-Planck like equations in a series of recent papers \cite{Mcmp,MMcmp,GMM,MS,EM*,CM*,FFP-IT,MM*,MQT*}. The proofs  are based on a splitting of the generator method as here and previously explained,  but the appropriate splitting are rather different for the two kinds of models. 
The operator $\AA_\eps$ is a multiplication ($0$-order) operator  for a Fokker-Planck equation while it is an integral ($-1$-order) operator for a Boltzmann equation.
More importantly, the fundamental and necessary regularizing effect is given by the action of the semigroup $S_{\BB_\eps}$ for the Fokker-Planck equation while it is given by the action of the operator $\AA_\eps$ for the Boltzmann equation. Let us underline here that in Section \ref{sec:DFFP-FFP}, we exhibit a new splitting for fractional diffusion Fokker-Planck operators (different from the one introduced in \cite{FFP-IT}) in the spirit of Boltzmann like operators (the operator $\AA_\eps$ is an integral operator whereas it was a multiplication operator in  \cite{FFP-IT} and in Section~\ref{sec:FFP-FP}). Our purpose is precisely to show that all these equations can be handled in the same framework, by exhibiting a suitable and compatible splitting 
\eqref{eq-intro:Splitting} which does not blow up  and such that the time indexed family of operators $\AA_\eps S_{\BB_\eps}$ (or some iterated 
convolution products of that one) have a good regularizing property which is uniform in the singular limit  $\eps \to 0$. 
  
\medskip\noindent
{\sl Probability interpretation.} The discrete and fractional Fokker-Planck equations are the evolution equations satisfied by the law of the stochastic process which is solution to the SDE
\beqn\label{eq:EDS}
dX_t = - X_t dt - d\LLL^\eps_t, 
\eeqn
where $\LLL^\eps_t$ is the Levy (jump) process associated to $k_\eps/\eps^2$ or $c_\eps/|z|^{d+2-\eps}$. 
For two trajectories $X_t$ and $Y_t$ to the above SDE associated to some initial data $X_0$ and $Y_0$, 
and $p \in [1,2)$, we have
$$
d |X_t - Y_t|^p = - p  |X_t - Y_t|^p dt,
$$
from which we deduce
$$
\EEE(|X_t - Y_t|^p) \le e^{-pt}  Â \EEE(|X_0 - Y_0|^p), \quad \forall \, t \ge 0.
$$
We fix now $Y_t$ as a stable process for the SDE \eqref{eq:EDS}. 
Denoting by $f_\eps(t)$ the law of $X_t$ and $G_\eps$ the law of  $Y_t$, we classically deduce 
the Wasserstein distance estimate 
\beqn\label{eq-intro:intro:WpftG}
W_p(f_\eps(t),G_\eps) \le e^{-t}  Â \, W_p(f_0,G_\eps), \quad \forall \, t \ge 0.
\eeqn
In particular, for $p=1$, the Kantorovich-Rubinstein Theorem says that \eqref{eq-intro:intro:WpftG} is equivalent to the estimate \beqn\label{eq-intro:intro:W1ftG}
\|Â f_\eps(t) - G_\eps\|_{(W^{1,\infty}(\R^d))'} \le e^{-t}  Â \, \|Â f_0- G_\eps\|_{(W^{1,\infty}(\R^d))'}, \quad \forall \, t \ge 0.
\eeqn
Estimates \eqref{eq-intro:intro:WpftG} and \eqref{eq-intro:intro:W1ftG} have to be compared with \eqref{eq-intro:intro:SLtoG}. Proceeding in a similar way as in \cite{MS,MM*}
it is likely that the spectral gap estimate \eqref{eq-intro:intro:W1ftG} can be extended (by ``shrinkage of the space") to a weighted Lebesgue space framework 
and then to get the estimate in Theorem~\ref{theo:RoughVersion} for any $a \in (-1,0)$.

%

\medskip\noindent
{\sl Singular kernel and other confinement term.} We also believe that a similar analysis can be handle with more singular kernels than the ones considered here, 
the typical example should be $k(z) = (\delta_{-1} + \delta_1)/2$ in dimension $d=1$, and for confinement term different from the harmonic confinement considered here,
including other forces or discrete confinement term. In order to perform such an analysis one could use some trick developed in \cite{MS} in order to handle the 
equal mitosis (which uses one more iteration of the convolution product of the time indexed family of operators $\AA_\eps S_{\BB_\eps}$). 

\medskip\noindent
{\sl Linearized and nonlinear equations.} We also believe that a similar analysis can be adapted to nonlinear equations. The typical example we have in mind
is the Landau grazing collision limit of the Boltzmann equation. One can expect to get an exponential trend of solutions to its associated Maxwellian equilibrium 
which is uniform with respect to the considered model (Boltzmann equation with  and without Grad's cutoff and Landau equation).

\medskip\noindent
{\sl Kinetic like models.} A more challenging issue would be to extend the uniform asymptotic analysis to the Langevin SDE or the kinetic Fokker-Planck equation by using some idea developed in \cite{CM*} which make possible to connect (from a spectral analysis point of view) the parabolic-parabolic Keller-Segel 
equation to the parabolic-elliptic Keller-Segel equation. The next step should be to apply the theory to the Navier-Stokes diffusion limit of the (in)elastic Boltzmann equation. 
These more technical problems will be investigated in next works.

\smallskip
\subsection{Outline of the paper}
Let us describe the plan of the paper. In each section, we treat a family of equations in a uniform framework, from a spectral analysis viewpoint with a semigroup approach. 
In Section~\ref{sec:DFP-FP}, we deal with the discrete and classical Fokker-Planck equations.
Section~\ref{sec:FFP-FP} is dedicated to the analysis of the fractional and classical Fokker-Planck equations. Finally, Section~\ref{sec:DFFP-FFP} is devoted to the study of the discrete and fractional Fokker-Planck equations.

\smallskip
\subsection{Notations}
For a (measurable) moment function $m : \R^d \to \R_+$, we define the norms
$$
\| f \|_{L^p(m)} := \|Â f \, m \|_{L^p(\R^d)}, \quad \|Â f \|^p_{W^{k,p}(m)} := \sum_{i=0}^k \|Â \partial^i f\|^p_{L^p(m)}, \quad k\ge1,
$$
and the associated weighted Lebesgue and Sobolev spaces $L^p(m)$ and $W^{k,p}(m)$, we denote $H^k(m) = W^{k,2}(m)$ for $k \ge1$.
We also use the shorthand $L^p_r$ and $W^{1,p}_r$ for the Lebesgue and Sobolev spaces $L^p(\nu)$ and $W^{1,p}(\nu)$ when the weight~$\nu$ is defined as
$\nu(x) = \langle x \rangle^r$, $\langle x \rangle := (1 + |x|^2)^{1/2}$.

We denote by $m$ a polynomial weight $m(x):=\langle x \rangle^q$ with $q>0$, the range of admissible $q$ will be specified throughout the paper. 

In what follows, we will use the same notation $C$ for positive constants that may change from line to line. Moreover, the notation $A \approx B$ shall mean that there exist two positive constants $C_1$, $C_2$ such that $C_1 A \le B \le C_2 A$. 

\medskip
\paragraph{\bf Acknowledgments.} The research leading to this paper was (partially) funded by the French ``ANR blanche'' project Stab: ANR-12-BS01-0019. The second author has been partially supported by the fellowship l'Or\'{e}al-UNESCO {\it For Women in Science}.

\bigskip 
 \section{From discrete to classical Fokker-Planck equation }
\label{sec:DFP-FP}
\setcounter{equation}{0}
\setcounter{theo}{0}


\smallskip

 In this section, we consider a kernel $k \in W^{2,1}(\R^d) \, \cap  \, L^1_3(\R^d)$ which is symmetric, i.e. $k(-x)=k(x)$ for any $x \in \R^d$, satisfies the normalization condition
\beqn \label{prop-DtoC:k}
\int_{\R^d} k(x) \, \begin{pmatrix} 1 \\ x \\ x \otimes x  \end{pmatrix} \, dx = \begin{pmatrix} 1 \\ 0 \\  2 I_d 
 \end{pmatrix},
\eeqn
 as well as the positivity condition:  there exist $\kappa_0, \rho > 0$ such that 
\beqn \label{prop-DtoC:kbisbis}
k \ge \kappa_0 \, \mathds{1}_{B(0,\rho)}.
\eeqn

\smallskip
We define $k_\eps(x) := 1/\eps^d k(x/\eps)$, $x \in \R^d$ for $\eps>0$, and we consider the discrete and classical Fokker-Planck equations 
\beqn \label{eq-DtoC:FPD} 
\left\{\begin{aligned}
&\partial_t f = {1 \over \eps^2} (k_\eps * f - f) + \hbox{div}(xf) =: \Lambda_\eps f, \quad \eps>0,\\
&\partial_t f = \Delta f + \hbox{div}(xf) =: \Lambda_0 f.
\end{aligned}
\right.
\eeqn

\smallskip
The main result of the section reads as follows. 

\begin{theo}\label{theo-DtoC:DtoC} Assume $r>d/2$ and consider a symmetric kernel $k$ belonging to $W^{2,1}(\R^d) \cap L^1_{2r_0+3}$ where $r_0> \max(r+d/2,5+d/2)$   which satisfies \eqref{prop-DtoC:k} and \eqref{prop-DtoC:kbisbis}. 

(1) For any $\eps > 0$, there exists a positive and unit mass normalized steady state $G_\eps \in L^1_r(\R^d)$ to the discrete Fokker-Planck equation 
\eqref{eq-DtoC:FPD}. 

(2) There exist explicit constants $a_0<0$ and $\eps_0>0$ such that for any $\eps \in [0,\eps_0]$, the semigroup $S_{\Lambda_\eps}(t)$ associated to
the discrete Fokker-Planck equation~\eqref{eq-DtoC:FPD} satisfies: for any $f \in L^1_r$ and any $a>a_0$, 
$$
\| S_{\Lambda_\eps}(t) f - G_\eps  \langle f \rangle  \|_{L^1_r} \leq C_a \, e^{at} \, \| f- G_\eps \langle f \rangle\|_{L^1_r}, \quad \forall \, t \ge0, 
$$
for some explicit constant $C_a\ge1$. In particular, the spectrum $\Sigma (\Lambda_\eps)$ of $\Lambda_\eps$ satisfies the separation property  $\Sigma (\Lambda_\eps) \cap D_{a_0} = \{0\}$ in $L^1_r$,  where we have denoted $D_\alpha := \{ \xi \in \R^d; \,\, \Re e \,\xi > \alpha \}$.  
\end{theo}

The method of the proof consists in introducing a suitable splitting of the operator $\Lambda_\eps$ as $\Lambda_\eps = \AA_\eps + \BB_\eps$, in establishing some dissipativity and regularity properties on 
$\BB_\eps$ and $\AA_\eps S_{\BB_\eps}$ and finally in applying the version  \cite{MS,Mbook*} of the Krein-Rutman theorem as well as  the perturbation theory developed in \cite{MMcmp,Granular-IT*,Mbook*}.

\smallskip
\subsection{Splitting of $\Lambda_\eps$}
Let us fix $\chi \in \DD(\R^d)$ radially symmetric and satisfying $\mathds{1}_{B(0,1)}Â \le \chi \le \mathds{1}_{B(0,2)}$. We define $\chi_R$ by $\chi_R(x) := \chi(x/R)$ for $R>0$ 
 and we denote $\chi_R^c := 1 - \chi_R$. 

\smallskip
For $\eps > 0$, we define the splitting $\Lambda_\eps = \AA_\eps+\BB_\eps$  with 
$$
\AA_\eps f := M \, \chi_R \, (k_\eps * f), 
$$
$$
\BB_\eps f := \left( {1 \over \eps^2} - M \right) (k_\eps * f - f) + M \, \chi_R^c \, (k_\eps * f -f) + \hbox{div}(xf) - M \, \chi_R \, f, 
$$
for some constants $M$, $R$ to be chosen later. Similarly, we define the splitting $\Lambda_0 = \AA_0+\BB_0$  with  $\AA_0 f := M \, \chi_R f$ and thus 
$\BB_0 f := \Lambda_0 f - M \, \chi_R f$ for some constants $M$, $R$ to be chosen later. 

%

\smallskip
\subsection{Uniform boundedness of $\AA_\eps$}
\begin{lem} \label{lem-DtoC:Abounded} For any $p\in [1,\infty]$, $s \ge 0$ and any weight function $\nu \ge 1$, the
 operator $\AA_\eps$ is bounded from $W^{s,p}$  into $W^{s,p}(\nu)$ with norm independent of $\eps$.
\end{lem}

\begin{proof} For any $f \in L^p(\nu)$, we have
$$
\begin{aligned}
\| \AA_\eps f \|_{L^p(\nu)}  
&\leq C \, \| k_\eps * f \|_{L^p} \leq C \, \|f\|_{L^p}.
\end{aligned}
$$
thanks to the Young inequality and because $\|k_\eps\|_{L^1} = \|k\|_{L^1} = 1$.  We conclude that  $\AA_\eps$ is bounded from $L^p$  into $L^p(\nu)$.  
The proof for the case $s>0$ is similar and it is thus skipped. 
\end{proof}

\smallskip
\subsection{Uniform dissipativity properties of $\BB_\eps$}

\begin{lem} \label{lem-DtoC:BdissipLp}
Consider $p\in [1,2]$ and  $q>d(p-1)/p$. Let us suppose that $k \in L^1_{pq+1}$.
For any $a>d(1-1/p)-q$, there exist $\eps_0>0$, $M \geq 0$ and $R \geq0$ such that for any $\eps \in [0, \eps_0]$, $\BB_\eps - a$ is dissipative in $L^p(m)$, or equivalently
\beqn\label{eq:BepsDissipL1}
\langle (\BB_\eps - a) f , \Phi'(f) \rangle_{L^p(m)} \le 0, \quad \forall \, f \in \DD(\R^d), \,\, \Phi(f) = |f|^p/p.
\eeqn
\end{lem}

\begin{proof} We split the operator in several pieces 
$$
\begin{aligned}
\BB_\eps f =  \left( {1 \over \eps^2} - M \right) \,  (k_\eps * f - f) &+ M \, \chi_R^c \, (k_\eps * f - f) \\
&+
 \hbox{div}(xf)  -  M \, \chi_R \, f =: \BB^1_\eps + ... + \BB^4_\eps,
\end{aligned}
$$
and we estimate each term
$$
T_i := \langle \BB^i_\eps f, \Phi'(f) \rangle_{L^p(m)} =  \int_{\R^d}  \left(\BB^i_\eps f \right) \,  (\hbox{sign} f)\, |f|^{p-1} \, m^p \, dx
$$
separately. From now on, we  consider $a>d(1-1/p)-q$, we fix $\eps_1>0$ such that  $M \leq 1/(2\eps_1^2)$ and we consider $\eps \in (0,\eps_1]$.
%
%
%
 \smallskip
We first deal with $T_1$. We observe that 
\bear \label{eq-DtoF:calculconvexBIS}
 (f(y) - f(x)) \, \hbox{sign}(f(x))\, |f|^{p-1}(x)  
 \le {1 \over p} (|f|^p(y) - |f|^p(x)),  
\eear
using the convexity of $\Phi$. We then compute 
$$
\begin{aligned}
T_1 &= \left( {1 \over \eps^2} - M \right) \int _{\R^d\times \R^d} k_\eps(x-y) \, (f(y)-f(x))  \, \Phi'(f(x)) \, m^p(x) \, dy \, dx \\
&\le  
{1 \over p}  \left( {1 \over \eps^2} - M \right) \int_{\R^d\times \R^d} \left(|f|^p(y) - |f|^p(x) \right) \, k_\eps(x-y) \,  m^p(x) \, dy \, dx \\
&= {1 \over p} \left( {1 \over \eps^2} - M \right) \int_{\R^d\times \R^d} \left(m^p(y) - m^p(x) \right) \, k_\eps(x-y) \,  |f|^p(x) \, dy \, dx, 
\end{aligned}
$$
where we have performed a change of variables to get the last equality. From a Taylor expansion, we have  
$$
m^p(y) -m^p(x) = (y-x) \cdot \nabla m^p(x) + \Theta(x,y),
$$
where 
$$
\begin{aligned}
 | \Theta(x,y) | &\le  
{1 \over 2}  \int_0^1 | D^2m^p(x+\theta(y-x)) (y-x,y-x) | \, d\theta \\
&\le C \, |x-y|^2 \, \langle x \rangle^{pq-2} \, \langle x-y \rangle^{pq-2},
\end{aligned}
$$
for some constant $C \in (0,\infty)$.  The term involving the gradient of $m^p$ gives no contribution because of~\eqref{prop-DtoC:k} and we thus obtain  
\beqn \label{eq-DtoC:T1}
\begin{aligned}
T_1 &\leq C \, \left( 1 - M \eps^2 \right) \, 
\int_{\R^d \times \R^d} k_\eps(x-y) \, { |x-y|^2 \over \eps^2}\, \langle x-y \rangle^{pq-2} \, dy \, |f|^p(x) \langle x \rangle^{pq-2} \, dx \\
&\leq C \, \int_{\R^d} |f|^p(x) \, \langle x \rangle^{pq-2} \, dx.
\end{aligned}
\eeqn

We now treat the second term $T_2$. Proceeding as above and thanks to \eqref{eq-DtoF:calculconvexBIS} again, we have 
$$
\begin{aligned}
T_2 &= \int_{\R^d \times \R^d} M \, \chi_R^c(x) \, k_\eps(x-y) \, (f(y) - f(x)) \, \Phi'(f(x)) \, m^p(x) \, dy \, dx  \\
&\le {M \over p} \int_{\R^d \times \R^d} k(z) \, \{Â \chi_R^c (x+ \eps z) \, m^p(x+ \eps z)  -\chi_R^c(x) \, m^p(x) \}Â \, dz \, |f(x)|^p  \, dy \\
\end{aligned}
$$
Using the mean value theorem 
$$
\chi_R^c(x+\eps z) = \chi_R^c(x) + \eps \,z \cdot \nabla \chi_R^c (x+ \theta \eps z), \, \, \,  m^p(x+\eps z) = m^p(x) + \eps z \cdot \nabla m^p(x + \theta' \eps z),
$$
for some $\theta, \theta' \in (0,1)$, and the estimates 
$$
|\nabla \chi_R^c|Â \le C_R \quad \text{and}  \quad |\nabla m^p(y+ \theta' \eps z)| \le C \, \langle y \rangle^{pq-1} \langle z \rangle^{pq-1},
$$ 
we conclude that 
\beqn \label{eq-DtoC:T2}
T_{2} \leq M \, C_R \, \eps \,  \int_{\R^d} |f|^p  \, m^p . 
\eeqn

As far as $T_3$ is concerned, we just perform an integration by parts:
\beqn \label{eq-DtoC:T3}
\begin{aligned}
T_3  
&= d \int_{\R^d} |f|^p \, m^p - {1 \over p} \int_{\R^d} |f|^p \, \hbox{div}(x \, m^p) \\
&= \int_{\R^d} |f(x)|^p  \, m^p(x) \, \left( d \left( 1 - {1 \over p}\right) - \frac{ q \, |x|^2 }{\langle x \rangle^2}Â \right) \, dx.
\end{aligned}
\eeqn
The estimates (\ref{eq-DtoC:T1}), (\ref{eq-DtoC:T2}) and (\ref{eq-DtoC:T3}) together give
$$
\begin{aligned}
\int _{\R^d}\BB_\eps f \, \Phi'(f) \, m^p &\leq \int_{\R^d} |f|^p \, m^p \, \left( C \, \langle x \rangle^{-2} + \frac{d}{p'} -\frac{ q \, |x|^2 }{\langle x \rangle^2}+ M \, C_R \, \eps - M \, \chi_R \right) \\
&= \int _{\R^d}  |f|^p \, m^p  \left( \psi^\eps_{R,p} - M \, \chi_R \right),
\end{aligned}
$$
where $p' = p/(p-1)$ and we have denoted 
\beqn \label{def:psitilde}
\psi^\eps_{R,p} (x):= C \, \langle x \rangle^{-2} + \frac{d}{p'} -\frac{ q \, |x|^2 }{\langle x \rangle^2} + M \, C_R \, \eps.
\eeqn
Because $\psi^\eps_{R,p}(x) \to d/p' -q$ when $\eps\to0$ and $|x|Â \to \infty$, we can thus choose $M \ge 0$, $R \ge 0$ and $\eps_0 \le \eps_1$ such that for any $\eps \in(0,\eps_0]$,
$$
\forall \, x \in \R^d, \quad \psi^\eps_{R,p}(x) \le a.
$$
As a conclusion, for such a choice of constants, we obtain \eqref{eq:BepsDissipL1}. We refer to~\cite{GMM,MM*} for the proof in the case $\eps=0$.
\end{proof}

\begin{lem} \label{lem-DtoC:BdissipHs}
Let $s \in \N$ and $q>d/2+s$. Assume that $k \in L^1_{2q+1}$. Then, for any $a>d/2-q+s$, there exist $\eps_0>0$, $M \geq 0$ and $R \geq0$ such that for any $\eps \in [0, \eps_0]$, $\BB_\eps - a$ is hypodissipative in $H^s(m)$. 
\end{lem}

\begin{proof}
The case $s=0$ is nothing but Lemma \ref{lem-DtoC:BdissipLp} applied with $p=2$. We now deal with the case $s=1$. We consider $f_t$ a solution to
$$
\partial_t f_t = \BB_\eps f_t.
$$
From the previous lemma, we already know that 
\beqn \label{eq-DtoC:evolutionf}
{1 \over 2}{ d \over dt} \| f_t \|^2_{L^2(m)} \leq  \int_{\R^d} f_t^2 \,Â m^2 \left(\psi^\eps_{R,2} - M \chi_R \right).
\eeqn
We now want to compute the evolution of the derivative of $f_t$:
$$
\partial_t \partial_x f_t = \BB(\partial_x f_t) + M \, \partial_x(\chi_R^c) \, (k_\eps*f_t -f_t) + \partial_x f_t,
$$
which in turn implies that 
$$
\begin{aligned}
\frac{1}{2} \frac{d}{dt} \| \partial_x f_t \|^2_{L^2(m)}
&= \int_{\R^d} (\partial_x f_t) \, \partial_t (\partial_x f_t) \, m^2\\
&= \int_{\R^d} (\partial_x f_t) \, \BB(\partial_x f_t) \, m^2 + \int_{\R^d} M \,\partial_x(\chi_R^c) \, (k_\eps*f_t) \, (\partial_x f_t) \, m^2 \\
&\quad- \int_{\R^d} M \, \partial_x(\chi_R^c) \, f_t \, (\partial_x f_t) \, m^2+ \int_{\R^d} (\partial_x f_t)^2 \, m^2 \\
&=: T_1 + T_2 + T_3+ T_4.
\end{aligned}
$$

Concerning $T_1$, using the proof of Lemma \ref{lem-DtoC:BdissipLp}, we obtain
\beqn \label{eq-DtoC:T1bis}
T_1 \leq \int_{\R^d} (\partial_x f_t)^2 \, m^2 \left( \psi^\eps_{R,2} - M \, \chi_R\right).
\eeqn

Then, to deal with $T_2$, we first notice that using Jensen inequality and \eqref{prop-DtoC:k}, we have
$$
\begin{aligned}
\|k_\eps * f\|^2_{L^2(m)}
&= \int_{\R^d} \left( \int_{\R^d} k_\eps(x-y) \, f(y) \, dy \right)^2 m^2(x) \, dx \\
&\leq \int_{\R^d \times \R^d } k_\eps(x-y) \, m^2(x) \, dx \, f^2(y) \, dy \\
&= \int_{\R^d \times \R^d} k(z) \, m^2(y + \eps z) \, dz\,  f^2(y) \, dy \\
&\leq C \, \int_{\R^d} k(z) \, m^2(z) \, dz \, \int_{\R^d} f^2 \, m^2.
\end{aligned}
$$
We thus obtain using that $k \in L^1_{2q}$:  
$$
\|k_\eps * f\|_{L^2(m)} \leq C \, \|f\|_{L^2(m)}.
$$
The term $T_2$ is then treated using Cauchy-Schwarz inequality, Young inequality and the fact that $|\partial_x(\chi_R^c)|$ is bounded by a constant depending only on $R$:
\beqn \label{eq-DtoC:T2bis}
\begin{aligned}
T_2 &\leq M \, C_R \, \|k_\eps *_x f_t \|_{L^2(m)} \|\partial_x f_t\|_{L^2(m)} \\
&\leq M \, C_R \,   \|f_t\|_{L^2(m)} \|\partial_x f_t\|_{L^2(m)} \\
&\leq M \, C_R \, K(\zeta) \|f_t\|^2_{L^2(m)} +  M \, C_R \, \zeta \|\partial_x f_t\|^2_{L^2(m)}
\end{aligned}
\eeqn
for any $\zeta>0$ as small as we want. 

The term $T_3$ is handled using an integration by parts and with the fact that $|\partial_x^2(\chi_R^c)|$ is bounded with a constant which only depends on $R$:
\beqn \label{eq-DtoC:T3bis}
T_3 = {M \over 2} \int_{\R^d} \partial^2_x(\chi_R^c) \, f_t^2 \, m^2 + {M \over 2} \int_{\R^d} \partial_x(\chi_R^c) \, f_t^2 \, \partial_x(m^2)  \leq M \, C_R \, \|f_t\|^2_{L^2(m)}.
\eeqn
%
%
Combining estimates \eqref{eq-DtoC:T1bis}, \eqref{eq-DtoC:T2bis} and \eqref{eq-DtoC:T3bis}, 
we easily deduce 
\beqn \label{eq-DtoC:evolutiondf}
\begin{aligned}
\frac{1}{2} \frac{d}{dt} \| \partial_x f_t \|^2_{L^2(m)} &\leq 
C_{R,M,\zeta} \int_{\R^d} f_t^2 \, m^2 \\
&\quad+ \int_{\R^d} (\partial_x f_t)^2 \, m^2 \, \left( \psi^\eps_{R,2} + M \, C_R \, \zeta + 1 - M \, \chi_R \right).
\end{aligned}
\eeqn
To conclude the proof in the case $s=1$, we introduce the norm
 $$
 \Nt f\Nt^2_{H^1(m)} := \|f\|^2_{L^2(m)} + \eta \, \|\partial_x f \|^2_{L^2(m)}, \quad \eta >0.
 $$
 Combining (\ref{eq-DtoC:evolutionf}) and (\ref{eq-DtoC:evolutiondf}), we get
 $$
 \begin{aligned}
 {1 \over 2}{d \over dt} \Nt f_t \Nt^2_{H^1(m)} 
 &\leq \int_{\R^d} f_t^2 \, m^2 \left( \psi^\eps_{R,2} + \eta \, C_{R,M,\zeta}- M \chi_R \right) \\
&\quad  + \eta \int_{\R^d} (\partial_x f_t)^2 \, m^2 \, \left( \psi^\eps_{R,2}  + M \, C_R \, \zeta + 1 - M \, \chi_R \right).
 \end{aligned}
 $$
 
Using the same strategy as in the proof of Lemma \ref{lem-DtoC:BdissipLp}, if $a>d/2-q+1$, we can choose $M$, $R$ large enough and $\zeta$, $\eps_0$, $\eta$ small enough such that we have on $\R^d$
$$
\psi^\eps_{R,2} + \eta \, C_{R,M,\zeta}- M \chi_R \leq a \quad \text{and} \quad 
\psi^\eps_{R,2} + M \, C_R \, \zeta + 1 - M \, \chi_R\leq a
$$
for any $\eps \in (0,\eps_0]$, which implies that 
$$
\frac{1}{2} \frac{d}{dt}\Nt f_t \Nt^2_{H^1(m)} \leq a \, \Nt f_t \Nt_{H^1(m)}^2. 
$$

 The higher order derivatives are treated with the same method introducing a similar modified $H^s(m)$ norm.
\end{proof}

\smallskip
\subsection{Uniform $\BB_\eps$-power regularity of  $\AA_\eps$.} \label{subsec:AisBpoweReg}
In this section we prove that  $\AA_\eps S_{\BB_\eps}$ and its iterated convolution products fulfill nice regularization and
growth estimates. 

We introduce the notation 
\beqn \label{eq-DtoC:Ieps}
I_\eps(f) := {1 \over{2\eps^2}}  \int_{\R^d \times \R^d} (f(x)-f(y))^2 \, k_\eps(x-y) \, dx \, dy.
\eeqn
\begin{lem} \label{lem-DtoC:Iepsilon}
There exists a constant $K>0$ such that for any $\eps>0$, the following estimate holds:
\beqn\label{eq-DtoC:kepsf}
\| \nabla (k_\eps * f )\|^2_{L^2} \leq K \, I_\eps(f).
\eeqn
\end{lem}
\begin{proof} {\sl Step 1. } We prove that the assumptions made on $k$ imply
\beqn \label{prop-DtoC:kter}
|\widehat{k} (\xi)|^2 \leq K \, \frac{1 - \widehat{k}(\xi)}{|\xi|^2}, \quad \forall \, \xi \in \R^d, 
\eeqn
for some constant $K>0$. On the one hand, we have $\widehat{k}(0)= 1$,  $\widehat{k}(\xi) \in \R$ because  $k$ is symmetric and $\widehat{k} \in C_0(\R^d)$ because $k \in L^1(\R^d)$.  Moreover, performing a Taylor expansion, using the normalization condition \eqref{prop-DtoC:k} and the fact that $k \in L^1_3(\R^d)$, we have 
$$
\widehat{k}(\xi)  = 1 - |\xi|^2 + \OO (|\xi|^3), \quad \forall \, \xi \in \R^d.
$$
We then deduce that \eqref{prop-DtoC:kter} holds with $K=1$ in a small ball $\xi \in B(0,\delta)$. On the other hand, for any $\xi \not=0$, we have 
\bean
\widehat{k}(\xi) 
&=& \int_{E_\xi} k(x) \, \cos (\xi \cdot x) \, dx  + \int_{E^c_\xi} k(x) \, \cos (\xi \cdot x) \, dx 
\\
&<& \int_{E_\xi} k(x) \, dx  + \int_{E^c_\xi} k(x)   \, dx  = 1 , 
\eean
where $E_\xi := \{ x \in \R^d; \, x \cdot \xi \in (0,\pi), \,\, |x| \le r \}$ so that $k(x) \cos (\xi \cdot x ) <  k(x)$ for any $x \in E_\xi$ from \eqref{prop-DtoC:kbisbis}. Together with the fact that $\widehat{k} \in C_0(\R^d)$, we deduce that $1 - \widehat{k} (\xi) \ge \eta > 0$ for any $\xi \in B(0,\delta)^c$. 
Last, because $k \in W^{1,1}(\R^d)$, we also have $|\xi|^2 \, |\widehat{k}(\xi)|^2Â = |\widehat{\nabla k}(\xi)|^2Â \le C$ for any $\xi \in \R^d$. We then deduce that \eqref{prop-DtoC:kter} holds with $K = C/\eta$ in the set $B(0,\delta)^c$. 

\Black

\smallskip\noindent
{\sl Step 2. } From the normalization condition \eqref{prop-DtoC:k}, we have  
 $$
 \begin{aligned}
 I_\eps(f) 
& = {1 \over{2\eps^2}}  \int_{\R^d \times \R^d} f^2(x) \, k_\eps(x-y) \, dx \, dy
 + {1 \over{2\eps^2}}  \int_{\R^d \times \R^d} f^2(y) \, k_\eps(x-y) \, dx \, dy \\
 &\quad - {1 \over{\eps^2}}  \int_{\R^d \times \R^d} f(x)f(y) \, k_\eps(x-y) \, dx \, dy \\
 &= {1 \over \eps^2} \left( \int_{\R^d} f^2 - \int_{\R^d} (k_\eps*f )\, f \right).
 \end{aligned}
 $$
As a consequence, using Plancherel formula and the identity $\widehat{k_\eps}(\xi) = \widehat{k}(\eps \, \xi)$, $\forall \, \xi \in\R^d$, we get
$$
 \begin{aligned}
 I_\eps(f) &= {1 \over \eps^2} \left( \int_{\R^d} \widehat{f}^2 - \int_{\R^d} \widehat{k_\eps} \, \widehat{f}^2 \right) = \int_{\R^d} \widehat{f}^2(\xi) \frac{1- \widehat{k}(\eps \xi)}{\eps^2} \, d\xi.
 \end{aligned}
$$
Then, we use again Plancherel formula to obtain 
\bean
  \|Â \partial_x (k_\eps *f)\|_{L^2}^2 
= 
 \| \FF (\partial_x (k_\eps *f)) \|_{L^2}^2 
= \int_{\R^d} |\xi|^2 \, \widehat{k}(\eps \xi)^2 \, \widehat{f}^2 .
\eean
We conclude to \eqref{eq-DtoC:kepsf} by using \eqref{prop-DtoC:kter}. 
 \end{proof}


We now introduce the following notation $\lambda := 1/(2K) >0$ and go into the analysis of regularization properties of the semigroup $\AA_\eps S_{\BB_\eps}(t)$. 
\Black

\begin{lem} \label{lem-DtoC:regularizationL2Hs}
Consider $s_1 < s_2 \in \N$ and $q>d/2+s_2$.  We suppose that $k \in L^1_{2q+1}$. 
 Let $M$, $R$ and $\eps_0$ so that the conclusion of Lemma~\ref{lem-DtoC:BdissipHs} holds in both spaces $H^{s_1}(m)$ and $H^{s_2}(m)$. Then, for any $a \in (\max\{ d/2-q+s_2, -\lambda\},0)$, there exists $n \in \N$ such that for any $\eps \in [0,\eps_0]$, we have the following estimate 
$$
 \| (\AA_\eps S_{\BB_\eps})^{(*n)}(t) \|_{H^{s_1}(m) \rightarrow H^{s_2}(m)} \leq C_a \, e^{at}, 
$$
for some constant $C_a>0$. 
\end{lem}

\begin{proof}
We first give the proof for the case $(s_1,s_2)=(0,1)$. We consider $a \in (\max\{ d/2-q+1, -\lambda\},0)$, $\alpha_0$ and $\alpha_1$ such that 
$a>\alpha_0>\alpha_1>\max\{ d/2-q+1, -\lambda\}$ and $f_t:= S_{\BB_\eps}(t) f$ with $f \in L^2(m)$, i.e. that satisfies 
$$
\partial_t f_t = \BB_\eps f_t, \quad f_0=f.
$$
From the proof of Lemma \ref{lem-DtoC:BdissipHs}, there exists $\eps_0$ such that for any $\eps \in (0,\eps_0]$, we have 
$$
\begin{aligned}
&\quad {1 \over 2} {d \over dt} \|f_t\|^2_{L^2(m)}\\
&\leq  - {1 \over 2} \left( {1 \over \eps^2} - M \right)\int_{\R^d \times \R^d} \left( f(y) - f(x) \right)^2 \, k_\eps(x-y)\,   m^2(x) \, dy \, dx 
  + \alpha_0\, \|f_t\|^2_{L^2(m)} \\
  &\leq -{1 \over 4\eps^2}  \int_{\R^d \times \R^d} \left( f(y) - f(x) \right)^2 \, k_\eps(x-y)\, dy \, dx   + \alpha_0 \, \|f_t\|^2_{L^2(m)} \\
  &\leq - {1 \over 2} \, I_\eps(f_t)  + \alpha_0 \, \|f_t\|^2_{L^2(m)} \\
  \end{aligned}
  $$
  where we have used that $M \leq 1/(2\eps^2)$ for any $\eps \in (0, \eps_0]$. Using Lemma \ref{lem-DtoC:Iepsilon}, we obtain 
  $$
  \begin{aligned}
 {d \over dt} \|f_t\|^2_{L^2(m)} 
&\leq - 2\lambda    \|k_\eps *_x f_t\|^2_{\dot{H}^1} + 2\alpha_0 \, \|f_t\|^2_{L^2(m)} \\
&\leq 2\alpha_0  \, \|k_\eps *_x f_t\|^2_{\dot{H}^1} + 2\alpha_0 \, \|f_t\|^2_{L^2(m)}.
\end{aligned}
$$
Multiplying this inequality by $e^{-2\alpha_0t}$, it implies that
$$
 {d \over dt} \left( \|f_t\|^2_{L^2(m)} \, e^{-2\alpha_0t} \right) \leq 2 \alpha_0 \, \|k_\eps *_x f_t\|^2_{\dot{H}^1} \, e^{-2\alpha_0t}
 $$
 and thus, integrating in time
$$
 \|f_t\|^2_{L^2(m)} \, e^{-2\alpha_0t} - 2\alpha_0 \int_0^t  \|k_\eps *_x f_s\|^2_{\dot{H}^1} e^{-2\alpha_0s} \, ds \leq \|f\|^2_{L^2(m)}.
 $$
In particular, we obtain
\beqn \label{ineq-DtoC:Iepsilon}
\int_0^\infty  \|k_\eps *_x f_s\|^2_{\dot{H}^1}   e^{-2\alpha_0s}  \,ds \leq -{1 \over 2\alpha_0} \,  \|f\|^2_{L^2(m)}.
\eeqn
We now want to estimate 
$$
\begin{aligned}
&\quad \int_0^\infty \|Â \AA_\eps S_{\BB_\eps}(s) f \|^2_{H^1(m)} \, e^{-2\alpha_0s} \, ds \\
&=  \int_0^\infty \|Â \AA_\eps f_s\|^2_{L^2(m)} \,e^{-2\alpha_0s}\, ds +  \int_0^\infty \|Â \partial_x \left( \AA_\eps f_s \right)\|^2_{L^2(m)}\, e^{-2\alpha_0s}\, ds \\
& \leq  \int_0^\infty \|Â \AA_\eps f_s\|^2_{L^2(m)} \,e^{-2\alpha_0s}\, ds+  \int_0^\infty \|Â M \partial_x(\chi_R) \, k_\eps *_x f_s\|^2_{L^2(m)}\,e^{-2\alpha_0s} \, ds\\
&\quad +  \int_0^\infty \|Â M \chi_R \, \partial_x(k_\eps *_x f_s)\|^2_{L^2(m)} \, e^{-2\alpha_0s}\, ds \\
&\quad =: I_1 + I_2 + I_3.
\end{aligned}
$$
Using dissipativity properties of $\BB_\eps$ and boundedness of $\AA_\eps$, we get
$$
I_1 \leq \int_0^\infty e^{2\alpha_1 s} e^{-2\alpha_0s} \, ds \, \|f\|^2_{L^2(m)} \leq C \, \|f\|^2_{L^2(m)}.
$$
We deal with $I_2$ using the fact that $M \partial_x(\chi_R)$ is compactly supported, Young inequality and dissipativity properties of $\BB_\eps$:
$$
\begin{aligned}
I_2 &\leq C \, \int_0^\infty \|k_\eps *_x f_s\|^2_{L^2} \, e^{-2\alpha_0s} \, ds \leq C \, \int_0^\infty \|f_s\|^2_{L^2} e^{-2\alpha_0s} \, ds \\
&\leq C \, \int_0^\infty e^{2\alpha_1 s} \, e^{-2\alpha_0s} \, ds \, \|f\|^2_{L^2(m)} \leq C \, \|f\|^2_{L^2(m)}.
\end{aligned}
$$
Finally, for $I_3$, we use \eqref{ineq-DtoC:Iepsilon} to obtain
$$
I_3 \leq  \int_0^\infty \|Â  k_\eps *_x f_s\|^2_{\dot{H}^1} \, e^{-2\alpha_0s}\, ds \leq C \, \|f\|^2_{L^2(m)}.
$$
All together, we have proved 
$$
\int_0^\infty \|\AA_\eps S_{\BB_\eps}(s) f \|^2_{H^1(m)} \, e^{-2\alpha_0s} \, ds \leq C \, \|f\|^2_{L^2(m)}.
$$
Consequently, using Cauchy-Schwarz inequality, we have 
\beqn \label{eq-DtoC:regL2H1}
\begin{aligned}
&\quad\left(\int_0^\infty \|\AA_\eps S_{\BB_\eps}(s) f \|_{H^1(m)} \, e^{-as} \, ds \right)^2\\
&\leq \int_0^\infty \|\AA_\eps S_{\BB_\eps}(s) f \|^2_{H^1(m)} \,e^{-2\alpha_0 s} \, ds \int_0^\infty  e^{-2 (a-\alpha_0)  s}\, ds \\
&\leq C \, \|f\|^2_{L^2(m)}.
\end{aligned}
\eeqn 
Fom the dissipativity of $\BB_\eps$ in $H^1(m)$ proved in Lemma~\ref{lem-DtoC:BdissipHs} and the fact that $\AA_\eps$ is bounded in $H^1(m)$, we also have 
$$
 \|\AA_\eps S_{\BB_\eps}(s)  \|_{H^1(m) \to H^1(m)} \, e^{-as} \leq C, \qquad \forall \, s \ge 0.
 $$
Using the two last estimates together, we deduce that for any $t \ge 0$
\bean
&&\|Â (\AA_\eps S_{\BB_\eps})^{(*2)}(t) f \|_{H^1(m)} \\
&&\quad\le \int_0^t  \|\AA_\eps S_{\BB_\eps}(t-s)  \|_{H^1(m) \to H^1(m)} \|\AA_\eps S_{\BB_\eps}(s) f \|_{H^1(m)} \, ds 
\\
&&\quad\le C \, e^{at}Â \int_0^\infty  e^{- a s} \,  \|\AA_\eps S_{\BB_\eps}(s) f \|_{ H^1(m)} \, ds 
\\
&&\quad\le C \, e^{at}Â  \,  \| f \|_{L^2 (m)}. 
\eean
We have thus proved
$$
\|Â (\AA_\eps S_{\BB_\eps})^{(*2)}(t)  \|_{L^2(m) \to H^1(m)}  \le C \, e^{at},
$$
which corresponds to the case $(s_1,s_2)=(0,1)$.

Using the same strategy, we can easily obtain that 
$$
 \int_0^\infty \|Â \AA_\eps S_{\BB_\eps}(s) f \|^2_{H^{s}(m)} \, e^{-2as} \, ds \leq C \, \|f\|^2_{H^{s-1}(m)},
 $$
for any $s \ge 2$, and then conclude the proof of the lemma in the case $\eps > 0$.  We refer to \cite{GMM,MM*} for the proof in the case $\eps=0$.
\end{proof}

\begin{lem} \label{lem-DtoC:regularizationL1L2}
Consider $q>d/2$,  $k \in L^1_{2q+1}$  and $M$, $R$, $\eps_0$ so that the conclusions of Lemma~\ref{lem-DtoC:BdissipLp} hold. Then, for any $a \in (-q,0)$, there exists $n \in \N$ such that the following estimate holds for any $\eps \in [0, \eps_0]$:
$$
\forall \, t \geq0, \quad  \|(\AA_\eps S_{\BB_\eps})^{(*n)} (t) \|_{\BBB(L^1(m), L^2(m))}  \leq C_a \, e^{at},
$$
\Black
for some constant $C_a>0$. 
\end{lem}

\begin{proof}
We first introduce the formal dual operators of $\AA_\eps$ and $\BB_\eps$:
$$
\AA_\eps^* \phi := k_\eps * (M \, \chi_R \, \phi), \quad 
\BB_\eps^* \phi := {1 \over \eps^2} ( k_\eps * \phi - \phi) - x \cdot \nabla \phi - k_\eps * (M \, \chi_R \phi).
$$
We use the same computation as the one used to deal with $T_1$ is the proof of Lemma \ref{lem-DtoC:BdissipLp} and Cauchy-Schwarz inequality:
$$
\begin{aligned}
\int_{\R^d} (\BB_\eps^* \phi) \, \phi
&\leq -{1 \over {2 \eps^2}} \int_{\R^d \times \R^d} k_\eps(x-y) \, (\phi(y) - \phi(x))^2 \, dy \, dx \\
&\quad + 
{1 \over {2\eps^2}} \int_{\R^d \times \R^d} (\phi^2(y) - \phi^2(x)) \, k_\eps(x-y) \, dy \, dx \\
&\quad + {d \over 2} \int_{\R^d} \phi^2 + \|k_\eps * (M \, \chi_R \, \phi)\|_{L^2} \, \|\phi\|_{L^2}.
\end{aligned}
$$
We then notice that the second term equals $0$ and we use Young inequality and the fact that $\|k_\eps\|_{L^1}=1$ to get 
$$
\begin{aligned}
\int_{\R^d} (\BB_\eps^* \phi) \, \phi &\leq -{1 \over {2 \eps^2}} \int_{\R^d \times \R^d} k_\eps(x-y) \, (\phi(y) - \phi(x))^2 \, dy \, dx \\
&\quad +
{d \over 2} \int_{\R^d} \phi^2 + {1 \over 2} \|M \, \chi_R \, \phi \|_{L^2}^2 + {1 \over 2} \|\phi \|_{L^2}^2  \\
&\leq - \, I_\eps(\phi) + C \, \int_{\R^d} \phi^2
\end{aligned}
$$
where $I_\eps$ is defined in \eqref{eq-DtoC:Ieps}.
We also have the following inequality:
$$
\begin{aligned}
 I_{\eps}(\chi_R \, \phi) &\leq 
{1 \over \eps^2} \int_{\R^d \times \R^d} k_\eps(x-y)\, \phi^2(x) \, (\chi_R(y) - \chi_R(x))^2 \, dy \, dx \\
&\quad + {1 \over \eps^2} \int_{\R^d \times \R^d} k_\eps(x-y)\, \chi_R^2(y) \, (\phi(y)-\phi(x))^2 \, dy \, dx  \\
&\leq C \, \|\nabla \chi_R\|_{\infty} \int_{\R^d} \phi^2 +2 I_\eps(\phi).
\end{aligned}
$$
If we denote $\phi_t := S_{\BB^*_\eps}(t) \phi$, we thus have
$$
{1 \over 2} {d \over dt} \|\phi_t\|^2_{L^2} \leq - {\lambda} \, \|k_\eps * (\chi_R \, \phi_t)\|^2_{\dot{H}^1} + b \, \|\phi_t\|^2_{L^2}, \quad b > 0.
$$
Multiplying this inequality by $e^{-bt}$, we obtain
$$
{d \over dt} \left( \|\phi_t\|^2_{L^2} \, e^{-bt} \right) \leq  - {2\lambda} \, \|k_\eps * (\chi_R \, \phi_t)\|^2_{\dot{H}^1} \, e^{-bt},Â \quad \forall \, t \geq 0, 
$$
and integrating in time, we get
\beqn \label{eq-DtoC:Iepsilon2}
\|\phi_t\|^2_{L^2} \, e^{-bt} + {2\lambda} \, \int_0^t\|k_\eps * (\chi_R \, \phi_s)\|^2_{\dot{H}^1} \, e^{-bs} \, ds \leq \|\phi\|^2_{L^2(m)}, \quad \forall \, t \ge 0.
\eeqn
We now estimate 
$$
\begin{aligned}
&\quad \int_0^t \|Â \AA_\eps^* \, S_{\BB_\eps^*}(s)\, \phi\|^2_{H^1} \, e^{-2bs} \, ds 
= \int_0^t \|Â \AA_\eps^* \, \phi_s \|^2_{H^1} \, e^{-2bs} \, ds \\
&= \int_0^t \| k_\eps * (M \, \chi_R \, \phi_s)\|^2_{L^2} \,e^{-2bs} \,  ds + \int_0^t \| k_\eps * (M \, \chi_R \, \phi_s) \|^2_{\dot{H}^1} \, e^{-2bs} \,  ds. 
\end{aligned}
$$
Using Young inequality and \eqref{eq-DtoC:Iepsilon2}, we conclude that 
$$
\int_0^\infty \|Â \AA_\eps^* \, S_{\BB_\eps^*}(t)\, \phi\|^2_{H^1} \, e^{-2bs} \, ds \leq C \, \|\phi\|^2_{L^2}.
$$
As in the proof of Lemma \ref{lem-DtoC:regularizationL2Hs}, for any $s \ge 1$, we can then establish that 
$$
  \|Â (\AA_\eps^* \, S_{\BB_\eps^*})^{(*2s)}(t) \|_{L^2 \to H^{s}} \le  C \, e^{b't}, \quad \forall \, t \ge 0, \,\, \forall \, \eps \in (0,\eps_0], 
$$
for some $b' \ge 0$, and by duality
$$
  \|Â (S_{\BB_\eps \AA_\eps})^{(*2s)}(t) \|_{H^{-s} \to L^2} \le  C \, e^{b't}, \quad \forall \, t \ge 0, \,\, \forall \, \eps \in (0,\eps_0].
$$
Taking $\ell>d/2$, so that we can use the continuous Sobolev embedding $L^1(\R^d) \subset H^{-\ell}(\R^d)$, we obtain
$$
 \|Â (S_{\BB_\eps \AA_\eps})^{(*2\ell)}(t) \|_{L^1 \to L^2} \le  C \, e^{b't}.
$$
Noticing next that 
$$
(\AA_\eps S_{\BB_\eps})^{(*(2\ell+1))} = \AA_\eps \, (S_{\BB_\eps} \AA_\eps)^{(*(2\ell))} * S_{\BB_\eps}
$$ 
and using the fact that $\AA_\eps$ is compactly supported combined with Lemma \ref{lem-DtoC:BdissipLp}, we get
\bean
&& \| (\AA_\eps S_{\BB_\eps})^{(*(2\ell+1))}(t) \|_{L^1(m) \to L^2(m)} 
\\
&&\quad\le
 \| \AA_\eps \|_{L^2 \to L^2(m)} \, 
 \bigl\{Â \| ( S_{\BB_\eps} \AA_\eps)^{(*(2\ell))} (\cdot) \|_{L^1  \to L^2 }  *_t   \| S_{\BB_\eps}(\cdot) \|_{L^1(m)  \to L^1 } \bigr\} (t) \\
 && \quad \le C \, e^{b''t}, 
\eean
for some $b'' \ge 0$. 
%
To conclude the proof, we use \cite[Lemma~2.17]{GMM}. Indeed, up to take more convolutions, we are able to recover a good rate in the last estimate. 
We refer to \cite{GMM,MM*} for the proof in the case $\eps=0$.
\end{proof}
\Black

%
%
\smallskip
\subsection{Convergences $\AA_\eps \rightarrow \AA_0$ and $\BB_\eps \rightarrow \BB_0$.}
\begin{lem} \label{lem-DtoC:convergence}
Consider $s \in \N$,   $q>0$ and $k \in L^1_{2q+3}$. The following convergences hold:
$$
\|\AA_\eps -\AA_0\|_{\BBB(H^{s+1}(m),H^s(m))} \xrightarrow[\eps \rightarrow 0]{}0 \quad \text{and} \quad \|\BB_\eps -\BB_0\|_{\BBB(H^{s+3}(m),H^s(m))} \xrightarrow[\eps \rightarrow 0]{}0.
$$
\end{lem}

\begin{proof}
\noindent {\it Step 1.} We first deal with $\AA_\eps$ in the case $s=0$. Using that $\chi \in \DD(\R^d)$ and $k \in L^1_1(\R^d)$, we have 
\bean
\|\AA_\eps f - \AA_0 f\|_{L^2(m)} &=& \|M \, \chi_R \, (k_\eps * f - f) \, m \|_{L^2} \leq C \, \|k_\eps * f - f\|_{L^2} 
\\
&=& C \, \|Â (\widehat{k_\eps} - 1) \, \widehat f \|_{L^2} \le C \, \eps \,  \|f\|_{H^1}.
\eean
Concerning the first derivative, writing that 
$$
\partial_x (\AA_\eps f - \AA_0 f) = M \, (\partial_x \chi_R) \, (k_\eps * f - f) + M \, \chi_R \, (k_\eps * \partial_x f - \partial_x f)
$$
and using that $\partial_x \chi_R$ is uniformly bounded 
as well as $\chi_R$, we obtain the result. We omit the details of the proof for higher order derivatives. 

\smallskip
\noindent {\it Step 2.} In order to prove the second part of the result, we just have to prove
$$
\|\Lambda_\eps -\Lambda_0\|_{\BBB(H^{s+3}(m),H^s(m))} \xrightarrow[\eps \rightarrow 0]{}0.
$$
Using \eqref{prop-DtoC:k}, we have 
$$
\Lambda_\eps f(x) - \Lambda_0 f(x) = {1 \over \eps^2} \int_{\R^d} k_\eps(x-y) (f(y) - f(x)) \, dy - \Delta f(x). 
$$
A Taylor expansion of $f$ gives
$$
\begin{aligned}
 f(y)-f(x) &=  (y-x) \cdot \nabla f(x) + {1 \over 2}Â  D^2f(x)(y-x,y-x) \\
 &\quad + {1 \over 2} \int_0^1 (1-s)^2 D^3f(x+s(y-x)) (y-x,y-x,y-x) \, ds. 
\end{aligned}
$$
We then observe that, because of \eqref{prop-DtoC:k}, the integral in the $y$ variable of the gradient term cancels and the contribution of the second term is precisely  $ \Delta f(x)$. We deduce that 
$$
 \Lambda_\eps f(x) - \Lambda_0 f(x) = {\eps \over 2} \int_{\R^d} k(z)  \int_0^1 (1-s)^2 D^3f(x+s\eps z) (z,z,z) \, ds \, dz.
 $$
Consequently, using Jensen inequality and the fact that  $k \in L^1_{2q+3}$, we get
$$
\begin{aligned}
&\quad  \|Â \Lambda_\eps - \Lambda_0\|_{L^2(m)} ^2 \\
&\leq C \, \eps^2 \int_{\R^d} \int_{\R^d}  k(z) \, |z|^3 \int_0^1 |D^3f(x+s\eps z)|^2  \,  m^2(x+s\eps z) \, m^2(s \eps z) \, ds \, dz \, dx   \\
&\leq C \, \eps^2  \, \|f\|_{H^3(m)} ^2\xrightarrow[\eps \rightarrow 0]{}0.
\end{aligned}
$$
This concludes the proof of the second part in the case $s=0$. The proof for $s > 0$ follows from the fact that the operator $\partial_x$ commutes with $\Lambda_\eps - \Lambda_0$. 
\end{proof}

\smallskip
\subsection{Spectral analysis}
\begin{lem} \label{lem-DtoC:Kato}
For any $\eps>0$, $\Lambda_\eps$ satisfies Kato's inequalities:
$$
\forall \, f \in D(\Lambda_\eps), \quad \Lambda_\eps \left(\beta(f)\right) \geq \beta'(f)\, (\Lambda_\eps f), \quad \beta(s) = |s|. 
$$
It follows that for any $\eps>0$, the semigroup associated to $\Lambda_\eps$ is positive in the sense that  $S_{\Lambda_\eps}(t) f \geq 0$ for any $t \geq 0$ if $f \in L^1(m)$ and $f\geq 0$. 
\end{lem}

\begin{proof}
First, we have
$$
\begin{aligned}
&\quad \hbox{sign} f(x) \, \Lambda_\eps f(x) \\
&= {1 \over \eps^2} \int_{\R^d} k_\eps(x-y) \, (f(y) - f(x)) \, dy \,\hbox{sign}f(x) \\
&\quad + d \, f(x) \, \hbox{sign} f(x) + x \cdot \nabla f(x) \,  \hbox{sign} f(x) \\
&\leq {1 \over \eps^2} \int_{\R^d} k_\eps(x-y) \, (|f|(y) - |f|(x)) \, dy+ d \, |f|(x) + x \cdot \nabla |f|(x) = \Lambda_\eps |f|(x),
\end{aligned}
$$
which ends the proof of the Kato inequality. 


We consider $f \leq 0$ and denote $f(t) := S_{\Lambda_\eps}(t)f$. We define the function $\beta(s) = s_+ = (|s|+s)/2$. 
Using  Kato's inequality, we have  $\partial_t \beta(f_t) \leq \Lambda_\eps \beta(f_t)$, and then 
$$
0 \leÂ \int_{\R^d} \beta(f_t)  \leq \int_{\R^d} \beta(f) = 0, \quad \forall \, t \ge 0, 
$$
from which we deduce $f_t \le 0$ for any $t \ge 0$. 
\end{proof}
The operator $-\Lambda_\eps$ satisfies the following form of the strong maximum principle. 

\begin{lem}\label{lem-DtoC:strongPM} Any nonnegative eigenfunction associated to the eigenvalue $0$ is positive. In other words, we have 
$$
f \in D(\Lambda_\eps), \quad \Lambda_\eps f = 0, \quad f \ge 0, \quad f \not= 0 \quad\hbox{implies}\quad f > 0.
$$
 \end{lem}

\begin{proof}Â We define 
$$
\CC f = {1 \over \eps^2}  Â k_\eps * f, \quad \DD f = x \cdot \nabla_x f + \lambda \, f, \quad \lambda := d - {1 \over \eps^2}
$$
and the semigroup 
$$
S_\DD(t) g := g(e^{t}  Â x) \, e^{\lambda t}
$$
with generator $\DD$. Thanks to the Duhamel formula
$$
S_{\Lambda_\eps} (t) = S_\DD(t) + \int_0^t S_\DD(s) \, \CC  S_\Lambda(t-s) \, ds,
$$
the eigenfunction $f$ satisfies 
\bean
f 
&=&  S_{\Lambda_\eps} (t) f = S_\DD(t) f + \int_0^t S_\DD(s) \, \CC S_{\Lambda_\eps}(t-s) f \, ds
\\
&\ge&    \int_0^t S_\DD(s) \, \CC f \, ds   \quad \forall \, t > 0.
\eean
By assumption, there exists $x_0 \in \R^d$ such that $f \not\equiv 0$ on $B(x_0,\rho/2)$. As a consequence, denoting $\vartheta := \|  Â f \|_{L^1(B(x_0,\rho/2))} > 0$, 
we have 
$$
\CC f \ge {\kappa_0 \, \vartheta \over \eps^2} \, \mathds{1}_{B(x_0,\rho/2)}, 
$$
and then 
\bean
f  \ge    {\kappa_0 \, \vartheta \over \eps^2} \, \sup_{t > 0}   \int_0^t e^{\lambda s} \, \mathds{1}_{B(e^{-s}  Â x_0,e^{-t}  Â  \rho/2)}  \, ds  \ge  \kappa_1   \mathds{1}_{B(x_0,\rho/4)}, \quad \kappa_1 > 0.  
\eean
Using that lower bound, we obtain 
$$
\CC f \ge \theta_d {\kappa_0 \, \kappa_{i-1} \over \eps^2} \, \mathds{1}_{B(x_0, u_i \rho)}, 
\quad\hbox{and then} \,\,  
f  \ge    \kappa_i  \mathds{1}_{B(x_0, v_i \rho)},
$$
with $i = 2$, $u_2 = 1$, $\kappa_2 > 0$,  $v_2 = 3/4$. Repeating once more the argument, we get the same  lower estimate with $i=3$, $u_3 = 7/4$, $\kappa_3 > 0$ and $v_3 = 3/2$. By an induction argument, we finally get $f > 0$ on $\R^d$.
\end{proof}

\medskip We are now able to prove Theorem~\ref{theo-DtoC:DtoC}.  We suppose that the assumptions of Theorem~\ref{theo-DtoC:DtoC} hold in what follows and thus consider $r>d/2$ and also $r_0 > \max(r+d/2, 5+ d/2)$.

\medskip\noindent
{\sl Proof of part (1) in Theorem~\ref{theo-DtoC:DtoC}. }Â 
Using Lemmas \ref{lem-DtoC:Abounded}-\ref{lem-DtoC:BdissipHs}-\ref{lem-DtoC:BdissipLp}, \ref{lem-DtoC:Kato}, \ref{lem-DtoC:strongPM} and the fact that $\Lambda_\eps^*1=0$, we can apply Krein-Rutman theorem which implies that for any $\eps>0$, there exists a unique $G_\eps >0$ such that $\|G_\eps\|_{L^1}=1$, $\Lambda_\eps G_\eps =0$ and $\Pi_\eps f = \langle f \rangle G_\eps$. It also implies that for any $\eps >0$, there exists $a_\eps<0$ such that in $X=L^1_r$ or $X=H^s_{r_0}$ for any $s \in \N$, there holds
$$
\Sigma(\Lambda_\eps) \cap  D_{a_\eps} = \{0\}
$$
and 
\beqn\label{eq-DtoC:EstimEps}
\forall \, t \ge 0, \quad \|Â S_{\Lambda_\eps}(t) f - \langle f \rangle \, G_\eps\|_{X} \leq e^{a t} \| f - \langle f \rangle \, G_\eps\|_{X}, \quad \forall \, a >a_\eps.
\eeqn

%

\medskip\noindent
{\sl Proof of part (2) in Theorem~\ref{theo-DtoC:DtoC}. }Â We now have to establish that estimate~\eqref{eq-DtoC:EstimEps} can be obtained uniformly in $\eps \in [0,\eps_0]$. In order to do so, we use a perturbation argument in the same line as in \cite{MMcmp,Granular-IT*} to prove that our operator $\Lambda_\eps$ has a spectral gap in $H^3_{r_0}$ which does not depend on $\eps$.
 
 First, we introduce the following spaces:
 $$
 X_1 := H^6_{r_0+1} \subset X_0 := H^3_{r_0} \subset X_{-1} := L^2_{r_0},
 $$
notice that $r_0>d/2+5$ implies that the conclusion of Lemma \ref{lem-DtoC:BdissipHs} is satisfied in the three spaces $X_i$, $i=-1,0,1$. 

 One can notice that we also have the following embedding
 $$
 X_1 \subset H^5_{r_0+1} \subset D_{L^2_{r_0}}(\Lambda_\eps) = D_{L^2_{r_0}}(\BB_\eps) \subset D_{L^2_{r_0}}(\AA_\eps) \subset X_0.
 $$
 
   
 We now summarize the necessary results to apply a perturbative argument (obtained thanks to Lemmas \ref{lem-DtoC:convergence}, \ref{lem-DtoC:Abounded}, \ref{lem-DtoC:BdissipLp}, \ref{lem-DtoC:BdissipHs} and~\ref{lem-DtoC:regularizationL2Hs} and from \cite{GMM,MM*}). 
 
 There exist $a_0<0$ and $\eps_0>0$ such that for any $\eps \in [0,\eps_0]$:
 \begin{enumerate}
 \item[(i)] For any $i=-1,0,1$, $\AA_\eps \in \BBB(X_i)$ uniformly in $\eps$. 
 \item[(ii)] For any $a>a_0$ and $\ell \ge 0$, there exists $C_{\ell,a}>0$ such that 
 $$
 \forall \, i=-1,0,1, \quad \forall \, t \ge 0, \quad \| S_{\BB_\eps} * (\AA_\eps S_{\BB_\eps})^{(*\ell)}(t)\|_{X_i \rightarrow X_i} \leq C_{\ell,a} \, e^{at}.
 $$
 \item[(iii)] For any $a>a_0$, there exist $n \geq 1$ and $C_{n,a}>0$ such that 
  $$ 
 \forall \, i=-1,0, \quad  \| (\AA_\eps S_{\BB_\eps})^{(*n)}(t) \|_{X_i \rightarrow X_{i+1}} \leq C_{n,a}\, e^{at}.
 $$ 
\item[(iv)] There exists a function $\eta(\eps) \xrightarrow[\eps \rightarrow 0]{}0$ such that
$$
\forall \, i=-1,0, \quad \|\AA_\eps - \AA_0 \|_{X_i \rightarrow X_i} \leq \eta(\eps) \quad \text{and} \quad 
\|\BB_\eps - \BB_0\|_{X_i \rightarrow X_{i-1}} \leq \eta(\eps).
$$
\item[(v)] $\Sigma (\Lambda_0) \cap  D_{a_0} = \{ 0 \}$ in spaces $X_i$, $i=-1,0,1$,
where $0$ is a one dimensional eigenvalue. 
\end{enumerate}

\smallskip

Using a perturbative argument as in \cite{Granular-IT*}, from the facts (i)--(v), we can deduce the following proposition:
\begin{prop}
There exist $a_0<0$ and $\eps_0>0$ such that for any $\eps \in [0,\eps_0]$, the following properties hold in $X_0=H^3_{r_0}$:
\begin{enumerate}
\item $\Sigma (\Lambda_\eps) \cap  D_{a_0} = \{0\}$;
\item for any $f \in X_0$ and any $a>a_0$, 
$$
\| S_{\Lambda_\eps}(t) f - G_\eps  \langle f \rangle  \|_{X_0} \leq C_a \, e^{at} \, \| f- G_\eps \langle f \rangle\|_{X_0}, \quad \forall \, t \ge0
$$
for some explicit constant $C_a>0$. 
\end{enumerate}
\end{prop}

To end the proof of Theorem \ref{theo-DtoC:DtoC}, we have to enlarge the space in which the conclusions of the previous Proposition hold. 
To do that, we use an extension argument (see \cite{GMM} or \cite[Theorem 1.1]{MMcmp}) and Lemmas~\ref{lem-DtoC:Abounded},~\ref{lem-DtoC:BdissipLp}-\ref{lem-DtoC:BdissipHs} and~\ref{lem-DtoC:regularizationL2Hs}-\ref{lem-DtoC:regularizationL1L2}. Our ``small'' space is $H^3_{r_0}$ and our ``large'' space is $L^1_r$ (notice that $r_0> r+ d/2$ implies the embedding $H^3_{r_0} \subset L^1_r$). 


\bigskip

\section{From fractional to classical Fokker-Planck equation} \label{sec:FFP-FP}
In this part, we denote $\alpha := 2-\eps \in (0,2]$ and we deal with the equations
\beqn \label{eq-FtoC:FFP}
\left\{
\begin{aligned}
&\partial_t f =- (-\Delta)^{\alpha/2}f + \hbox{div}(xf) = \Lambda_{2-\alpha} f =: \LL_{\alpha} f, \quad \alpha \in (0,2) \\
&\partial_t f = \Delta f + \hbox{div}(xf) = \Lambda_0 f =: \LL_2 f.
\end{aligned}
\right.
\eeqn
We here recall that the fractional Laplacian $\Delta^{\alpha/2}f$ is defined for a Schwartz function $f$ through the integral formula \eqref{eq-intro:opfrac}.
Moreover, the constant $c_\alpha$ in~\eqref{eq-intro:opfrac} is chosen such that 
$$
 {c_\alpha \over 2} \, \int_{|z| \leq 1} \frac{z_1^2}{|z|^{d+\alpha}}=1 , 
$$
which implies that $c_\alpha\approx (2-\alpha)$. By duality, we can extend the definition of the fractional Laplacian to the following class of functions:
$$
 \left\{ f \in L^1_{\text{loc}}(\R^d) , \, \, \int_{\R^d} |f(x)| \, \langle x \rangle^{-d-\alpha} \, dx < \infty \right\}.
 $$
In particular, one can define $(-\Delta)^{\alpha/2} m$ when $q<\alpha$.

We recall that the equation 
$\partial_t f = \LL_\alpha f$ admits a unique equilibrium of mass $1$ that we denote~$G_\alpha$ (see \cite{GI} for the case $\alpha<2$). Moreover, if $\alpha<2$, one can prove that $G_\alpha(x) \approx \langle x \rangle^{-d-\alpha}$ (see \cite{FFP-IT}) and for $\alpha=2$, we have an explicit formula $G_2(x) = (2\pi)^{-d/2} e^{-|x|^2/2}$. The main result of this section reads:
\begin{theo} \label{theo-FtoC:FFP-FP}
Assume $\alpha_0 \in (0,2)$ and $q<\alpha_0$. There exists an explicit constant $a_0<0$ such that for any $\alpha \in [\alpha_0,2]$, the semigroup $S_{\LL_\alpha}(t)$ associated to the fractional Fokker-Planck equation~\eqref{eq-FtoC:FFP} satisfies: for any $f \in L^1_q$, any $a>a_0$ and any $\alpha \in [\alpha_0,2]$,
$$
\|S_{\LL_\alpha}(t) f - G_\alpha \langle f \rangle\|_{L^1_q} \le C_a e^{at} \|f - G_\alpha \langle f \rangle\|_{L^1_q}
$$
for some explicit constant $C_a\ge 1$. In particular, the spectrum $\Sigma(\LL_\alpha)$ of $\LL_\alpha$ satisfies the separation property $\Sigma(\LL_\alpha) \cap D_{a_0} = \{ 0\}$ in $L^1_q$ for any $\alpha \in [\alpha_0,2]$. 
\end{theo}

\subsection{Exponential decay in $L^2(G_\alpha^{-1/2})$}
We recall a result from \cite{GI} which establishes an exponential decay to equilibrium for the semigroup $S_{\LL_\alpha}(t)$ in the small space $L^2(G_\alpha^{-1/2})$. 
\begin{theo} \label{theo-FtoC:GI} There exists a constant $a_0<0$ such that for any $\alpha \in (0,2)$, 

\begin{enumerate}
\item[(1)] in $L^2(G_\alpha^{-1/2})$, there holds $\Sigma(\LL_\alpha) \cap D_{a_0} = \{0\}$;
\item[(2)] the following estimate holds: for any $a>a_0$,
$$
\|S_{\LL_\alpha}(t) f - G_\alpha \langle f \rangle \|_{L^2(G_\alpha^{-1/2})} \leq e^{at} \, \| f - G_\alpha \langle f \rangle \|_{L^2(G_\alpha^{-1/2})}, \quad \forall \, t \ge 0. 
$$
\end{enumerate}
\end{theo}
%

\subsection{Splitting of $\LL_\alpha$ and uniform estimates. }Â 
The proof is based on the splitting of the operator $\LL_\alpha$ as $\LL_\alpha = \AA  + \BB_\alpha$ where $\AA$ is the multiplier operator  $\AA  f:= M \, \chi_R f $, for some $M,R>0$ to be chosen later,  and an extension argument taking advantage of the already known exponential decay in $L^2(G_\alpha^{-1/2})$. 

\smallskip
As a straightforward consequence of the definition of $\AA$, we get the following estimates. 

\begin{lem} \label{lem-FtoC:Abounded}
Consider $s \in \N$  and $p \ge 1$. The operator is uniformly bounded in $\alpha$ from $W^{s,p}(\nu)$ to $W^{s,p}$ with $\nu = m$ or $\nu = G_\alpha^{-1/2}$. 
\end{lem}  

We next establish that $\BB_\alpha$ enjoys uniform dissipativity properties.
\begin{lem} \label{lem-FtoC:Bdissip}
For any $a>-q$, there exist $M >0$ and $R>0$ such that for any $\alpha \in [\alpha_0,2]$, $\BB_\alpha -a$ is dissipative in $L^1(m)$. 
\end{lem}
\begin{proof}
We just have to adapt the proof of Lemma 5.1 from \cite{FFP-IT} taking into account the constant $c_\alpha$. Indeed, we have 
$$
\int_{\R^d} \left(\LL_\alpha f \right) \hbox{sign} f \, m \leq \int_{\R^d} |f| \, m \left( {I_\alpha(m) \over m} - {{x \cdot \nabla m} \over m} \right).
$$ 
We can then show that thanks to the rescaling constant $c_\alpha$, $I_\alpha(m)/m$ goes to $0$ at infinity uniformly in $\alpha \in [\alpha_0,2)$. As a consequence, if $a>-q$, since $(x \cdot \nabla m)/m$ goes to $-q$ at infinity, one may choose $M$ and $R$ such that for any $\alpha \in [\alpha_0,2)$, 
$$
{I_\alpha(m) \over m} - {{x \cdot \nabla m} \over m} - M \, \chi_R \leq a, \quad \text{on} \, \,\R^d,
$$
which gives the result.
\end{proof}

\begin{lem} \label{lem-FtoC:Bdissipbis}
For any $a>a_0$ where $a_0$ is defined in Theorem \ref{theo-FtoC:GI}, $\BB_\alpha -a$ is dissipative in~$L^2(G_\alpha^{-1/2})$. 
\end{lem}
\begin{proof}
The proof also comes from \cite[Lemma~5.1]{FFP-IT}. 
\end{proof}

We finally establish that $\AA S_{\BB_\alpha}$ enjoys some uniform regularization properties. 
\begin{lem} \label{lem-FtoC:regularization}
There exist some constants $b \in \R$ and $C>0$ such that for any $\alpha \in [\alpha_0,2]$, the following estimates hold:
$$
\forall \, t \geq 0, \quad \|S_{\BB_\alpha}(t)\|_{\BBB(L^1,L^2)} \leq C \, \frac{e^{bt}}{t^{d/2 \alpha_0}}.
$$
As a consequence, we can prove that for any $a>\max(-q,a_0)$, $\alpha \in[\alpha_0,2]$, 
\beqn \label{eq-FtoC:regbis}
\forall \, t \geq 0, \quad \|(\AA \, S_{\BB_\alpha})^{(*n)}(t)\|_{\BBB(L^1(m), L^2(G_\alpha^{-1/2}))} \leq C \, e^{at}.
\eeqn

\end{lem}

\begin{proof} We do not write the proof for the case $\alpha =2$, for which we refer to~\cite{GMM,MM*}.

\noindent \textit{Step 1.} The key argument to prove this regularization property of $S_{\BB_\alpha}(t)$ is the Nash inequality.
For $\alpha \in [\alpha_0,2)$, from the proof of \cite[Lemma~5.3]{FFP-IT}, we obtain that there exist $b \geq0$ and $C>0$ such that for any $\alpha \in [\alpha_0,2)$,
$$
\forall \, t \geq 0, \quad \|S_{\BB_\alpha}(t) f \|_{L^2} \leq C \, \frac{e^{bt}}{t^{d/(2\alpha_0)}} \, \|f\|_{L^1}.
$$

\smallskip
\noindent \textit{Step 2.} Using that $\AA$ is compactly supported, we can write 
$$
\|\AA S_{\BB_\alpha}(t) f \|_{L^2(m)} \leq C \, \| S_{\BB_\alpha}(t) f \|_{L^2} \leq C \, \frac{e^{bt}}{t^{d/(2\alpha_0)}} \,  \|f\|_{L^1}.
$$
Using the same method as in \cite{GMM}, we can first deduce that there exists $\ell_0 \in \N$, $\gamma \in [0,1)$ and $K \in \R$ such that for any $\alpha \in [\alpha_0,2]$,
$$
\|(\AA S_{\BB_\alpha})^{(* \ell_0)}(t) f \|_{L^2(G_\alpha^{-1/2})} \leq C \, \frac{e^{bt}}{t^\gamma} \,  \|f\|_{L^1(m)}.
$$ 
We  next conclude that \eqref{eq-FtoC:regbis} holds using \cite[Lemma~2.17]{GMM} together with Lemmas \ref{lem-FtoC:Bdissip} and~ \ref{lem-FtoC:Abounded}. 
\end{proof}

\smallskip
\subsection{Spectral analysis} Before going into the proof of Theorem \ref{theo-FtoC:FFP-FP}, let us notice that we can make explicit the projection $\Pi_\alpha$ onto the null space $\NN(\LL_\alpha)$ through the following formula: $\Pi_\alpha f= \langle f \rangle \, G_\alpha$. Moreover, since the mass is preserved by the equation $\partial_t f = \LL_\alpha f$, we can deduce that $\Pi_\alpha (S_{\LL_\alpha}(t) f) = \Pi_\alpha f$ for any $t \geq 0$.
 
\begin{proof}[Proof of Theorem \ref{theo-FtoC:FFP-FP}]
We apply  \cite[Theorem~2.13]{GMM} for each $\alpha \in [\alpha_0,2]$ because combining Theorem \ref{theo-FtoC:GI} with Lemmas \ref{lem-FtoC:Abounded}, \ref{lem-FtoC:Bdissip}, \ref{lem-FtoC:Bdissipbis} and \ref{lem-FtoC:regularization}, we can check the assumptions of the theorem are satisfied. 
\end{proof}
 
\bigskip

\section{From discrete to fractional Fokker-Planck equation} \label{sec:DFFP-FFP}
Let us fix $\alpha \in (0,2)$. We consider the equations 
\beqn \label{eq-DtoF:FPFD}
\left\{\begin{aligned}
&\partial_t f = k_\eps * f  - \|k_\eps\|_{L^1}  f + \hbox{div}_x (xf) =: \Lambda_\eps f, \quad \eps>0, \\
&\partial_t f = -(-\Delta )^{\alpha/2} f + \hbox{div}_x (xf) =: \Lambda_0 f,
\end{aligned}
\right.
\eeqn
where 
$$
k_\eps(x) := \mathds{1}_{\eps \le |x| \le 1/\eps} \, k_0(x) + \mathds{1}_{|x|<\eps} \, k_0(\eps), \quad k_0(x) := |x|^{-d-\alpha}.
$$
Notice that 
\beqn \label{prop-DtoF:keps2}
\forall \, x \in \R^d \setminus \{0\}, \quad k_\eps(x) \nearrow k_0(x) \quad \text{as} \quad \eps \to 0.
\eeqn

We here recall that for $\alpha \in (0,2)$, the fractional Laplacian on Schwartz functions is defined through the formula \eqref{eq-intro:opfrac}. Since $\alpha$ is fixed in this part, we can get rid of the constant $c_\alpha$ and consider that it equals $1$. 
The main theorem of this section reads: 
\begin{theo} \label{theo-DtoF:DtoF}
Assume  $0<r<\alpha/2$.

(1) For any $\eps > 0$, there exists a positive and unit mass normalized steady state $G_\eps \in L^1_r(\R^d)$ to the discrete fractional Fokker-Planck equation 
\eqref{eq-DtoF:FPFD}. 

(2) There exist an explicit constant $a_0<0$ and a constant $\eps_0>0$ such that for any $\eps \in [0,\eps_0]$, the semigroup $S_{\Lambda_\eps}(t)$ associated to
the discrete and fractional Fokker-Planck equations  \eqref{eq-DtoF:FPFD} satisfies: for any $f \in L^1_r$ and any $a>a_0$, 
$$
\| S_{\Lambda_\eps}(t) f - G_\eps  \langle f \rangle  \|_{L^1_r} \leq C_a \, e^{at} \, \| f- G_\eps \langle f \rangle\|_{L^1_r} \quad \forall \, t \ge0,
$$
for some explicit constant $C_a\ge1$. In particular, the spectrum $\Sigma (\Lambda_\eps)$ of $\Lambda_\eps$ satisfies the separation property  $\Sigma (\Lambda_\eps) \cap  D_{a_0} = \{0\}$ in $L^1_r$. 
\end{theo}

The method of the proof is similar to the one of Section \ref{sec:DFP-FP}. We introduce a suitable splitting $\Lambda_\eps = \AA_\eps + \BB_\eps$, establish some dissipativity and regularity properties on 
$\BB_\eps$ and $\AA_\eps S_{\BB_\eps}$ and apply the Krein-Rutman theory revisited in~\cite{MS,Mbook*}. However, let us emphasize that we introduce a new splitting for the fractional operator (a different one from Section \ref{sec:FFP-FP} and from \cite{FFP-IT}) and we also develop a new perturbative argument in the same line as \cite{MMcmp,Granular-IT*,Mbook*} but with some less restrictive assumptions on the operators $\AA_\eps$ and $\BB_\eps$, requiring that they are fulfilled only on the limit operator (i.e. for $\eps=0$).

\smallskip
\subsection{Splittings of $\Lambda_\eps$} For any $0<\beta<\beta'$, as previously, we introduce $\chi_\beta (x) := \chi(x/\beta)$, $\chi_\beta^c := 1- \chi_\beta$; we also define $\chi_{\beta, \beta'} := \chi_{\beta'}- \chi_{\beta}$ 
and introduce the function $\xi_\beta$ defined on $\R^d \times \R^d$ by $\xi_\beta(x,y) := \chi_\beta(x) + \chi_\beta(y) - \chi_\beta(x) \chi_\beta(y)$ and $\xi_\beta^c:=1 - \xi_\beta$. 
We denote $I_0(f) := -(-\Delta)^{\alpha/2}f$ and $I_\eps(f) := k_\eps * f- \|k_\eps\|_{L^1}  f$ for $\eps>0$. We split these operators into several parts: for any $\eps\ge0$,
\beqn \label{eq-DtoF:splittingIeps}
\begin{aligned}
I_\eps(f) (x) 
&= \int_{\R^d} k_\eps(x-y) \, \chi_\eta (x-y) \, (f(y) - f(x)- \chi(x-y) (y-x) \cdot \nabla f(x)) \, dy \\
&\quad +\int_{\R^d} k_\eps(x-y)\, \chi_L^c (x-y) \,  (f(y) - f(x)) \,dy \\
&\quad + \int_{\R^d} k_\eps(x-y)\, \chi_{\eta,L} (x-y) \,  (f(y) - f(x)) \,\xi_R^c(x,y) \, dy \\
&\quad - \int_{\R^d} k_\eps(x-y) \,  \chi_{\eta,L} (x-y) \,\xi_R(x,y) \, dy \, f(x) \\
&\quad + \int_{\R^d} k_\eps(x-y) \, \chi_{\eta,L} (x-y) \,\xi_R(x,y) f(y) \, dy \\
&=: \BB_\eps^1 f + \BB_\eps^2 f + \BB_\eps^3 f + \BB_\eps^4 f + \AA_\eps f.
\end{aligned}
\eeqn
where the constants $\eta \in [\eps,1]$, $R>0$ and $0<L \le 1/\eps$ will be chosen later. One can notice that given the facts that $\eta \ge \eps$ and $L \le 1/\eps$, we have for any $\eps>0$, $\AA_\eps = \AA_0 =: \AA$. 
Finally, we denote for any~$\eps \ge0$,
$$
\BB^5_\eps f = \hbox{div}(xf) \quad \text{and} \quad \BB_\eps f = \BB_\eps^1 f + \BB_\eps^2 f + \BB_\eps^3 f + \BB_\eps^4 f + \BB^5_\eps f.
$$ 
\smallskip
\subsection{Convergence $\BB_\eps \rightarrow \BB_0$.}
\begin{lem} \label{lem-DtoF:convergence}
Consider $p \in(1,\infty)$ and $q \in (0,\alpha/p)$. The following convergence holds:
$$
\|\BB_\eps -\BB_0\|_{\BBB(W^{s+2,p}(m),W^{s,p}(m))} \le \eta_1(\eps) \xrightarrow[\eps \rightarrow 0]{}0, \quad s=-2,0.
$$
\end{lem}
\begin{proof}
 Let us notice that $\BB_\eps - \BB_0=\Lambda_\eps-\Lambda_0$.

\noindent {\it Step 1.} We first consider the case $s=0$ and we introduce the notation $k_{0,\eps}:= k_0 - k_\eps$. We compute
$$
\begin{aligned}
&\quad \|\Lambda_\eps f - \Lambda_0 f\|^p_{L^p(m)} \\
&= \int_{\R^d} \left| \int_{\R^d} k_{0,\eps}(z)  \, (f(x+z) - f(x) - \chi(z) z \cdot \nabla f(x)) \, dz \right|^p \, m^p(x) \, dx \\
&\le C \, \int_{\R^d} \left| \int_{|z| \le 1} k_{0,\eps}(z)  \, (f(x+z) - f(x)- \chi(z) z \cdot \nabla f(x)) \, dz \right|^p \, m^p(x) \, dx \\
&\quad +  C \, \int_{\R^d} \left| \int_{|z| \ge 1} k_{0,\eps}(z)  \, (f(x+z) - f(x)- \chi(z) z \cdot \nabla f(x)) \, dz \right|^p \, m^p(x) \, dx\\
&=: T_1+T_2.
\end{aligned}
$$
To deal with $T_1$, we perform a Taylor expansion of $f$ of order $2$ and we use that 
$\chi(z)=1$ if $|z| \le 1$, in order to get
$$
T_1 \le C \, \int_{\R^d} \, \left( \int_{|z|Â \le 1}  k_{0,\eps}(z) \, |z|^2 \int_0^1 (1-s) \, |D^2f(x + s z)| \, ds \, dz \right)^p \, m^p(x) \, dx.
$$
From H\"older inequality applied with the measure $\mu_\eps(dz) := \mathds{1}_{|z| \le 1}\, k_{0,\eps}(z) \, |z|^2 \, dz$, we have 
$$
\begin{aligned}
T_1 
&\le C \, \left(\int_{\R^d} \mu_\eps(dz)\right)^{p/p'} \,  \int_{\R^d\times \R^d}  \left(\int_0^1   |D^2f(x + s z)| \, ds \right)^p  \mu_\eps (dz) \, m^p(x) \, dx
\end{aligned}
$$
where $p' = p/(p-1)$ is the H\"older conjugate exponent of $p$. 
Using now Jensen inequality, 
we get
$$
\begin{aligned}
T_1 &\le C \, \left(\int_{\R^d} \mu_\eps(dz)\right)^{p/p'} \,  \int_{\R^d\times \R^d}  \int_0^1  |D^2f(x + s z)|^p \, ds  \, \mu_\eps (dz) \, m^p(x) \, dx \\
& \le C \, \left(\int_{\R^d} \mu_\eps(dz)\right)^{p} \int_{\R^d} |D^2f(x)|^p \, m^p(x) \, dx , 
\end{aligned}
$$
with 
$$
\int_{\R^d} \mu_\eps(dz) = \int_{|z|Â \le 1} k_{0,\eps}(z) \, |z|^{2} \, dz
\xrightarrow[\eps \rightarrow 0]{} 0
$$
by Lebesgue dominated convergence theorem. 
To treat $T_2$, we first notice that the term involving $\nabla f(x)$ gives no contribution, because $k_{0,\eps} \chi  \equiv 0$ for $\eps \in (0,1/2)$, so that  performing  similar computations as for $T_1$, we have 
$$
\begin{aligned}
T_2 
&\le C \, \int_{\R^d} \left| \int_{|z| \ge 1} k_{0,\eps}(z)  \, (f(x+z) - f(x)) \, dz \right|^p \, m^p(x) \, dx \\
&\le C \,\left(\int_{|z| \ge 1} k_{0,\eps}(z) \, dz\right)^{p/p'}  \\
&\qquad \int_{\R^d}  \int_{|z|Â \ge 1} |k_{0,\eps}(z)| (|f|^p(x+z) +|f|^p(x)) \, dz \,m^p(x) \, dx\\
&\le C \,  \left(\int_{|z|Â \ge 1} k_{0,\eps}(z) \,m^p(z) \, dz \right)^p\, \int_{\R^d} |f|^p(x) \, m^p(x) \, dx,
\end{aligned}
$$
with 
$$
\int_{|z|Â \ge 1} k_{0,\eps}(z) \,m^p(z) \, dz \xrightarrow[\eps \rightarrow 0]{} 0
$$
by the Lebesgue dominated convergence theorem again.
As a consequence, we obtain
$$
\|(\Lambda_\eps - \Lambda_0)(f)\|_{L^p(m)} \le \eta(\eps) \|f\|_{W^{2,p}(m)}, \quad  \eta(\eps) \xrightarrow[\eps \rightarrow 0]{} 0.
$$

 \noindent {\it Step 2.} 
We now consider the case $s=-2$, and we recall that by definition
$$
\begin{aligned}
\| \Lambda_\eps f - \Lambda_0 f\|_{W^{-2,p}(m)} &=
\sup_{\|\phi\|_{W^{2,p'}} \le 1} \int_{\R^d} f \, (\Lambda_\eps - \Lambda_0)^*(\phi \, m) \\
&= \sup_{\|\phi\|_{W^{2,p'}} \le 1} \int_{\R^d} f \, (\Lambda_\eps - \Lambda_0)(\phi \, m)
\end{aligned}
$$
where $p'=p/(p-1)$ and because $(\Lambda_\eps - \Lambda_0)^* = \Lambda_\eps - \Lambda_0$ (where $\Lambda^*$ stands for the formal dual operator of~$\Lambda$). For sake of simplicity, we introduce the notation 
\beqn \label{eq:psinu}
\TT_{\nu}(x,y) := \nu(y) - \nu(x) - \nabla \nu (x) \cdot (y-x) \, \chi(x-y).
\eeqn
We then estimate the integral in the right hand side of the previous equality:
$$
\begin{aligned}
\int_{\R^d} f \, (\Lambda_\eps - \Lambda_0)(\phi \, m)
&= \int_{\R^d} \frac{(\Lambda_\eps - \Lambda_0)(\phi \, m)}{m} \, f \, m \\
&\le \| (\Lambda_\eps - \Lambda_0)(\phi \, m) /m \|_{L^{p'}} \, \|f \|_{L^p(m)}.
\end{aligned}
$$
Moreover, 
\beqn \label{eq-DtoF:I(phim)}
\begin{aligned}
&\quad (\Lambda_\eps - \Lambda_0)(\phi \, m)(x)  = (I_\eps - I_0)(\phi \, m)(x) \\
&= (I_\eps - I_0)(\phi)(x) \, m(x) + \int_{\R^d} k_{0,\eps}(z) \, \phi(x+z) \, \TT_m(x,x+z) \, dz \\
&\quad + \int_{\R^d} k_{0,\eps}(z) \, \chi(z) \, z \cdot \nabla m(x) \, (\phi(x+z) - \phi(x)) \, dz.
\end{aligned}
\eeqn
We deduce that 
$$
\begin{aligned}
& \|Â (\Lambda_\eps - \Lambda_0)(\phi \, m) /m \|^{p'}_{L^{p'}}
 \le C \, \Bigg( \| (I_\eps - I_0)(\phi) \|^{p'}_{L^{p'}} \\
& +  \int_{\R^d} \frac{1}{m^{p'}(x)} \left| \int_{\R^d} k_{0,\eps}(z)  \phi(x+z)\, \TT_m(x,x+z) \,  dz \right|^{p'}  dx \\
& + \int_{\R^d}  \frac{1}{m^{p'}(x)} \left| \int_{\R^d} k_{0,\eps}(z) \, \chi(z) \, z \cdot \nabla m(x) \, (\phi(x+z) - \phi(x)) \, dz \right|^{p'} \, dx \Bigg) \\
&=: C \, (J_1 + J_2 + J_3).
\end{aligned}
$$
To deal with $J_1$, we use the step 1 of the proof which gives us 
$$
 \| (I_\eps - I_0)(\phi) \|_{L^{p'}} \le \eta(\eps) \|\phi\|_{W^{2,p'}}, \quad \eta(\eps) \xrightarrow[\eps \rightarrow 0]{} 0.
 $$
 The term $J_2$ is split into two parts:
 $$
 \begin{aligned}
 J_2 &\le C \Bigg( \int_{\R^d} \frac{1}{m^{p'}(x)} \left| \int_{|z| \le 1} k_{0,\eps}(z)  \phi(x+z)  \, \TT_m(x,x+z) \,  dz \right|^{p'}  dx  \\
&\quad  +  \int_{\R^d} \frac{1}{m^{p'}(x)} \left| \int_{|z| \ge 1} k_{0,\eps}(z)  \phi(x+z) \, \TT_m(x,x+z) \,dz \right|^{p'} dx \Bigg) \\
&=: J_{21} + J_{22}.
 \end{aligned}
 $$
 We first notice that for $|z|\le 1$, 
 $$
 \TT_m(x,x+z) = \int_0^1 (1-\theta) D^2 m(x+\theta z)(z,z) \, d\theta, 
 $$
 which implies that 
 $$
 J_{21}
 \le C   \int_{\R^d} \frac{1}{m^{p'}(x)} \left( \int_0^1 \int_{|z| \le 1} k_{0,\eps}(z) |z|^2 |D^2 m(x+\theta z)|   |\phi|(x+z)   d\theta dz \right)^{p'} dx .
 $$
 Since $0<q<2$, $|D^2m| \le C$ and $1/m^{p'} \le C$ in $\R^d$, we thus deduce using H\"older inequality and a change of variable, 
 $$
 \begin{aligned}
 J_{21}
& \le C \, \left(\int_{|z| \le 1} k_{0,\eps}(z) \, |z|^2 \, dz \right)^{p'}\, \|\phi\|^{p'}_{L^{p'}} \quad \text{with} \quad \int_{|z| \le 1} k_{0,\eps}(z) \, |z|^2 \, dz \xrightarrow[\eps \rightarrow 0]{} 0.
\end{aligned}
 $$
 Concerning $J_{22}$, we use $|z \chi(z)| \le C$ for any $|z| \ge 1$ and $|\nabla m| \le C \, m$ in $\R^d$, and we obtain that $J_{22}$ is bounded from above by
 $$
 \begin{aligned}
 & C   \int_{\R^d} \frac{1}{m^{p'}(x)} \left( \int_{|z| \ge 1} k_{0,\eps}(z)  |\phi|(x+z)  \left(m(x+z) +m(x) + |\nabla m(x)| \right)  dz \right)^{p'}  dx \\
 &\le C   \int_{\R^d} \left( \int_{|z| \ge 1} k_{0,\eps}(z) \, |\phi|(x+z)  \, m(z) \, dz \right)^{p'}  dx , 
 \end{aligned}
 $$
 which implies, using H\"{o}lder inequality and a change of variable,
 $$
  \begin{aligned}
 &J_{22} 
 \le C \,  \left(\int_{|z|Â \ge 1} k_{0,\eps}(z) \, m^p(z) \, dz\right)^{p'} \,  \| \phi\|^{p'}_{L^{p'}}\\
 &\qquad \text{with} \quad \int_{|z| \ge 1} k_{0,\eps}(z) \,m^p(z) \, dz \xrightarrow[\eps \rightarrow 0]{} 0.
 \end{aligned}
 $$
 Finally, we handle $J_3$ performing a Taylor expansion of $\phi$: 
 $$
 \phi(x+z) - \phi(x) = \int_0^1 (1-s) \, \nabla \phi(x+ s z)   \cdot z\, ds
 $$
 which implies, using that $|\nabla m|^{p'}/m^{p'} \in L^\infty(\R^d)$, H\"{o}lder inequality and a change of variable,
 $$
 \begin{aligned}
 J_3 
 &\le \left(\int_{\R^d}  \frac{|\nabla m|^{p'}(x)}{m^{p'}(x)} \left( \int_{|z|\le 2} k_{0,\eps}(z)\,|z|^2 \int_0^1 |\nabla \phi|(x+s z) \, ds\, dz \right)^{p'} \, dx \right)^{1/p'} \\
&\le C \, \int_{|z| \le 2} k_{0,\eps}(z) \,|z|^2 \, dz \, \| \nabla \phi \|_{L^{p'}}
\quad \text{with} \quad \int_{|z| \le 2} k_{0,\eps}(z) \,|z|^2 \, dz \xrightarrow[\eps \rightarrow 0]{} 0.
\end{aligned}
 $$
 As a consequence, we obtain that 
 $$
  \|Â (\Lambda_\eps - \Lambda_0)(\phi \, m) /m \|_{L^{p'}} \le \eta(\eps) \|\phi\|_{W^{2,p'}}, \quad \eta(\eps) \xrightarrow[\eps \rightarrow 0]{} 0,
$$
which concludes the proof.
\end{proof}

\smallskip
\subsection{Regularization properties of $\AA_\eps$}
\begin{lem} \label{lem-DtoF:Abounded}
For any $p \in (1,\infty)$, $(s,t)=(-2,0)$ or $(0,2)$, the operator $\AA_\eps = \AA_0 =\AA$ defined in~\eqref{eq-DtoF:splittingIeps} by 
$$
\AA f = \int_{\R^d} k_0(x-y) \, \chi_{\eta,L} (x-y) \,\xi_R(x,y) f(y) \, dy 
$$
is bounded from $W^{s,p}$ to $W^{t,p}(\nu)$ for any weight function $\nu$.
\end{lem}

\begin{proof}
First, one can notice that 
\beqn \label{eq-DtoF:ineqcaract}
\begin{aligned}
\xi_R(x,y) \, \chi_{\eta,L}(x-y)
&\le  \left(\chi_R(x) + \chi_R(y)\right)\, \chi_{\eta,L}(x-y) \\
&\le  \left(\mathds{1}_{|x| \le 2R} + \mathds{1}_{|y| \le 2R}\right) \, \mathds{1}_{\eta \le |x-y|Â \le 2L} \\
&\le 2 \, \mathds{1}_{\eta \le |x-y|Â \le 2L} \, \mathds{1}_{|x|Â \le 2(R+L)} \, \mathds{1}_{|y|Â \le 2(R+L)},
\end{aligned}
\eeqn
the proof is hence immediate in the case $s=t=0$ using Young inequality:
$$
 \|\AA f\|_{L^p(\nu)}
\le C \, \|\AA f \|_{L^p} \le \| k_0 \, \mathds{1}_{\eta \le |\cdot| \le 2L} \|_{L^1} \, \|f\|_{L^p}.
$$

We now deal with the case $(s,t)=(0,2)$. First, we have for $\ell=1,2$
$$
\partial^\ell_x (\AA f)(x) = \sum_{i+j+k=\ell} \int_{\R^d} \partial^i_x (k_0(x-y)) \, \partial^j_x (\chi_{\eta,L}(x-y)) \, \partial^k_x (\xi_R(x,y)) \, f(y) \, dy,
$$
and for any $(i,j,k)$ such that $i+j+k=\ell$, 
$$
\begin{aligned}
&\quad | \partial^i_x (k_0(x-y)) \, \partial^j_x (\chi_{\eta,L}(x-y)) \, \partial^k_x (\xi_R(x,y)) | \\
&\le C \, | \partial^i_x (k_0(x-y)) | \,\mathds{1}_{\eta \le |x-y| \le 2L} \, \mathds{1}_{|x| \le 2(R+L)}.
\end{aligned}
$$
As a consequence, for $\ell = 0,1,2$,
$$
\| \partial_x^\ell (\AA f)\|_{L^p(\nu)} \le \sum_{i=0}^2 \| \partial_x^i k_0 \, \mathds{1}_{\eta \le |\cdot|\le2L} \|_{L^1} \, \|f\|_{L^p},
$$
which concludes the proof in the case $(s,t)=(0,2)$.

Finally,  arguing by duality, we have
$$
\begin{aligned}
\|\AA f\|_{L^p(\nu)} &\le   C \, \sup_{\|\phi\|_{L^{p'}} \le 1} \, \int_{\R^d} (\AA f) \, \phi 
= C \, \sup_{\|\phi\|_{L^{p'}} \le 1} \, \int_{\R^d} (\AA \phi) \, f \\
&\le C \, \sup_{\|\phi\|_{L^{p'}} \le 1} \|f\|_{W^{-2,p}} \, \|\AA \phi \|_{W^{2,p'}}
\le 
C \,  \|f\|_{W^{-2,p}},
\end{aligned}
$$
which proves the estimate in the case $(s,t)=(-2,0)$. 
\end{proof}

\smallskip
\subsection{Dissipativity properties of $\BB_\eps$ and $\BB_0$}
\begin{lem} \label{lem-DtoF:BdissipLp}
Consider $p \in [1,2]$ and $q \in (0,\alpha/p)$. For any $a>d(1-1/p)-q$, there exist $\eps_1>0$, $\eta>0$, $L > 0$ and $R >0$ such that for any $\eps \in [0, \eps_1]$, $\BB_\eps - a$ is dissipative in $L^p(m)$.
 \end{lem}

\begin{proof}
We consider $a>d(1-1/p)-q$ and we estimate for $i=1, \dots,5$ the integral $\int_{\R^d}\left( \BB_\eps^i f\right) \,(\hbox{sign} f) \, |f|^{p-1}\, m^p$. 

We first deal with $\BB_\eps^1$ in both cases $\eps>0$ and $\eps=0$ simultaneously noticing that for any $\eps \ge 0$,
$$
\BB_\eps^1 f (x)= \int_{\R^d} (k_\eps \, \chi_\eta) (x-y) \, (f(y)-f(x) - (y-x) \cdot \nabla f(x) ) \, dy. 
$$
Then, using \eqref{eq-DtoF:calculconvexBIS}, we have 
$$
\begin{aligned}
&\quad \int_{\R^d} \left(\BB_\eps^1 f\right)  (\hbox{sign} f) \, |f|^{p-1} \, m^p \\
&\leq {1 \over p} \int_{\R^d \times \R^d} \left( |f|^p(y)-|f|^p(x) - (y-x) \cdot \nabla |f|^p(x)\right)   (k_\eps \, \chi_\eta) (x-y)  dy \, m^p(x) \, dx \\
&= {1 \over p} \int_{\R^d \times \R^d} (m^p(y) - m^p(x) - (y-x) \cdot \nabla m^p(x)) \,  (k_\eps \, \chi_\eta) (x-y)\, dy \, |f|^p(x) \, dx.
\end{aligned}
$$
Using a  Taylor expansion of order $2$ and that  $pq<\alpha<2$, we get
$$
\begin{aligned}
& \int_{\R^d} (m^p(y) - m^p(x)- (y-x) \cdot \nabla m^p(x)) \,  (k_\eps \, \chi_\eta) (x-y) \, dy \\
&= \int_{\R^d} \int_0^1 (1-\theta) \, D^2 m^p(x + \theta z))(z,z) \,  (k_\eps \, \chi_\eta) (z) \, d\theta \, dz \\
&\leq C \, \int_{|z|\leq 2\eta}  |z|^2 \, k_0(z) \, dz,
\end{aligned}
$$
and thus 
$$
\begin{aligned}
&\int_{\R^d} \left(\BB_\eps^1 f\right) (\hbox{sign}  f) \, |f|^{p-1} \, m^p \leq \kappa_{\eta} \int_{\R} |f|^p \, m^p \\
&\qquad \text{with} \quad \kappa_{\eta} \approx \int_{|z|\leq2\eta} k_0(z) \, |z|^2 \, dz\xrightarrow[\eta \rightarrow 0]{} 0.
\end{aligned}
$$

Concerning $\BB^2_\eps$, we also treat the case $\eps>0$ and $\eps=0$ in a same time using~\eqref{eq-DtoF:calculconvexBIS}:
$$
\begin{aligned}
&\quad \int_{\R^d} \left(\BB_\eps^2 f \right) (\hbox{sign}f) \, |f|^{p-1} \, m^p \\
&\leq  {1 \over p} \int_{\R^d \times \R^d} k_\eps(x-y) \, \left( |f|^p(y) - |f|^p(x)\right) \, \chi_L^c(x-y) \, m^p(x) \, dy \, dx \\
&=  {1 \over p}  \int_{\R^d \times \R^d} k_\eps(x-y) \, \left( m^p(y) - m^p(x)\right) \,  \chi_L^c(x-y) \, |f|^p(x) \, dy \, dx.
\end{aligned}
$$
We now use the fact that the function $s \mapsto s^{pq/2}$ is $pq/2$-H\"{o}lder continuous since $pq/2 < \alpha/2 \leq 1$ to obtain 
\beqn \label{eq-DtoF:calculm}
\begin{aligned}
&\quad |m^p(x) - m^p(y)| 
\leq C \, \left|ÂÃÂÂÂ |x| - |y|ÂÃÂÂÂ \right|^{pq/2} \, \left(|x|+|y|\right)^{pq/2} \\
&\leq C \, |x-y|^{pq/2} \, \min\left(\left(|x| + |x-y| + |x| \right)^{pq/2}, \left(|y| + |x-y| + |y| \right)^{pq/2}\right)\\
&\leq C \, \left(\min\left(|x-y|^{pq/2} |x|^{pq/2},|x-y|^{pq/2} |y|^{pq/2}\right)+ |x-y|^{pq} \right)  \\
&\leq C \, \langle x-y \rangle^{pq} \, \min \left(\langle x \rangle^{pq/2}, \langle y \rangle^{pq/2} \right).
\end{aligned}
\eeqn
We deduce that
$$
\begin{aligned}
&\quad \int_{\R^d} \left(\BB_\eps^2 f\right)  (\hbox{sign} f) \, |f|^{p-1} \, m^p
\le C \, \int_{|z| \ge L} k_0(z) \, m^p(z) \, dz \, \int_{\R^d} |f|^p(x) \, \langle x \rangle^{pq/2} \, dx\\
&\le \kappa_L \, \int_{\R^d} |f|^p \, m^p, \quad \text{with} \quad \kappa_L \approx \int_{|z| \ge L} k_0(z) \, m^p(z) \, dz \xrightarrow[L \rightarrow +\infty]{} 0.
\end{aligned}
$$

We now handle the third term $\BB_\eps^3$ first using inequality \eqref{eq-DtoF:calculconvexBIS} again:
$$
\begin{aligned}
&\quad \int_{\R^d} \left(\BB_\eps^3 f\right)  (\hbox{sign} f) \, |f|^{p-1} \, m^p \\
&\le {1 \over p} \, \int_{\R^d \times \R^d} k_\eps(x-y) \, \chi_{\eta,L}(x-y) \,\xi_R^c(x,y) \, (|f|^p(y) - |f|^p(x)) \, m^p(x) \, dy \, dx \\
&= {1 \over p} \,  \int_{\R^d \times \R^d} k_\eps(z) \, \chi_{\eta,L}(z) \, \xi_R^c(y+z,y) \, |f|^p(y) \, (m^p(y+z) - m^p(y)) \, dy \, dz. 
\end{aligned}
$$
We then use the Taylor-Lagrange formula which gives us the existence of $\theta \in (0,1)$ such that 
$$
m^p(y+z) = m^p(y) + z \cdot \nabla m^p(y+ \theta z).
$$
Notice that there exists $C_L>0$ depending on $L$ such that $|\nabla m^p(y+ \theta z)| \le C_L \, \langle y \rangle^{pq-1}$ for any $y \in \R^d$, $|z| \le 2L$. We hence obtain 
$$
\begin{aligned}
&\quad \int_{\R^d} \left(\BB_\eps^3 f\right)  (\hbox{sign} f) \, |f|^{p-1} \, m^p  \\
&\le C_L \,  \int_{\R^d \times \R^d} k_\eps(z)\,|z| \, \chi_{\eta,L}(z) \, \xi_R^c(y+z,y) \, |f|^p(y) \, \langle y \rangle^{pq-1} \, dy \, dz \\
&\le C_L \,  \int_{\R^d \times \R^d} k_\eps(z)\,|z| \,  \chi_{\eta,L}(z) \,\chi_R^c(y) \, |f|^p(y) \, {{m^p(y)} \over \langle y \rangle} \, dy \, dz \\
&\le C_L \, \int_{\eta \le |z| \le 2L} k_0(z) \, |z| \, dz \, \int_{|y| \ge 2R} |f|^p(y) \, {{m^p(y)} \over \langle y \rangle} \, dy, 
\end{aligned}
$$
which leads to 
$$
\begin{aligned}
\int_{\R^d} \left(\BB_\eps^3 f\right)  (\hbox{sign} f) \, |f|^{p-1} \, m^p 
\le C_{\eta,L} \, \int_{\R^d} |f|^p(y) \, {{m^p(y)} \over R} \, dy.
\end{aligned}
$$
As a consequence, we obtain
$$
\int_{\R^d} \left(\BB_\eps^3 f\right)(\hbox{sign} f) \, |f|^{p-1} \, m^p  \leq \kappa_R \, C_{\eta,L} \, \int_{\R^d} |f| \, m \,\,\, \text{with} \,\,\, \kappa_R \approx {1 \over R} \xrightarrow[R \rightarrow +\infty]{} 0.
$$
We estimate the term involving $\BB^4_\eps$ using that $\xi_R(x,y)  \geq \chi_R(x)$,
and we get
$$
\int _{\R^d} \left(\BB^4_\eps f\right)  (\hbox{sign} f) \, |f|^{p-1} \, m^p \leq - \int_{2\eta \leq |z| \leq L} k_\eps(z)  \, dz \, \int_{|x| \leq R}  |f|^p\, m^p.
$$

Finally, using integration by parts, we have
$$
\begin{aligned}
&\quad \int_{\R^d} \left(\BB_\eps^5 f\right) (\hbox{sign} f) \, |f|^{p-1} \, m^p \\
&= \int_{\R^d} |f|^p(x) \, m^p(x) \left( d\left(1-{1\over p} \right) - \frac{x \cdot \nabla m^p(x)}{p\, m^p(x)} \right) \, dx. 
\end{aligned}
$$
Gathering all the previous estimates and denoting
$$
\begin{aligned}
\psi_{\eta,L,R}^\eps(x) &:= \kappa_\eta + \kappa_L + \kappa_R \, C_{\eta,L} - \int_{2\eta \leq |z| \leq L} k_\eps(z) \, dz \, \mathds{1}_{|x| \leq R} \\ &\quad- \left( d\left(1-{1\over p} \right) - \frac{x \cdot \nabla m^p(x)}{p\, m^p(x)} \right),
\end{aligned}
$$
we obtain
$$
\int_{\R^d} \left(\BB_\eps f\right) (\hbox{sign} f) \, |f|^{p-1} \, m^p \leq
\int_{\R^d} \psi_{\eta,L,R}^\eps(x) \,|f|^p(x) \, m^p(x) \, dx .
$$

First, since $\varphi_m : x \mapsto d(1-1/p) - {x \cdot \nabla m^p(x)}/{p\, m^p(x)}$ is a continuous function, we can bound it by above by a constant $C_R$ depending on $R$ on $\{ |x| \le R \}$ for any $R>0$. We denote $\ell := d(1-1/p) - q$ which is the limit of $\varphi_m$ as $|x| \rightarrow \infty$. One can also notice that $A^\eps_{\eta,L} :=  \int_{2\eta \leq |z| \leq L} k_\eps(z) \, dz \rightarrow \infty$ as $\eps \rightarrow 0$ and $\eta \rightarrow 0$. We first choose $\eps_1>0$, $\eta \ge \eps_1$, $L\le 1/\eps_1$ and $R>0$, so that we have
$$
|x| \ge R \Rightarrow \varphi_m(x) \le {{a+\ell} \over 2} \quad \text{and} \quad \kappa_\eta + \kappa_L + \kappa_R \, C_{\eta,L} \le {{a-\ell}\over 2}.
$$
Up to make decrease the value of $\eta$, we can then choose $\eps_0<\eps_1$ such that for any $\eps \in [0,\eps_0]$,
$$
\kappa_\eta + \kappa_L + \kappa_R \, C_{\eta,L} + C_R - A^\eps_{\eta,L} \le a.
$$
As a conclusion, for this choice of constants, for any $x \in \R^d$ and $\eps \in [0,\eps_0]$, we have $ \psi_{\eta,L,R}^\eps(x) \le a$, which yields the result. 
\end{proof}

\begin{lem} \label{lem-DtoF:BdissipHs}
Consider $q \in (0,\alpha/2)$. There exists $b \in \R$ such that for any $s \in \N$, $\BB_0 - b$ is hypodissipative in $H^s(m)$.
\end{lem}

\begin{proof}
\noindent {\it Step 1.} We first treat the case $s=0$. We write $\BB_0 = \Lambda_0 - \AA_0$ and we compute 
$$
\begin{aligned}
\int_{\R^d}  (\BB_0 f) \, f \, m^2 &= \int_{\R^d} (\Lambda_0 f) \, f \, m^2 - \int_{\R^d} (\AA_0 f) \, f \, m^2 \\
&= \int_{\R^d} I_0(f) \, f \, m^2 + \int_{\R^d} \hbox{div}(xf) \, f \, m^2 - \int_{\R^d} (\AA_0 f) \, f \, m^2 \\
&=: T_1 + T_2+ T_3.
\end{aligned}
$$
Concerning $T_1$, we have 
$$
\begin{aligned}
&T_1=\\& \int_{\R^d \times \R^d} k_0(x-y) \, (f(y) - f(x) - \chi(x-y) \, (y-x) \cdot \nabla f(x)) f(x) \, m^2(x) \, dy \, dx \\
&= -{1 \over 2}  \int_{\R^d \times \R^d} k_0(x-y) \, (f(y) - f(x))^2 \, dy \, m^2(x) \, dx + {1 \over 2} \int_{\R^d} f^2 \, I_0(m^2).
\end{aligned}
$$
Since one can prove that $I_0(m^2)/m^2$ goes to $0$ at infinity (cf Lemma 5.1 from \cite{FFP-IT}) and is thus bounded in $\R^d$, we can deduce that there exists $C \in \R_+$ such that 
$$
T_1 \le -{1 \over 2}  \int_{\R^d \times \R^d} k_0(x-y) \, (f(y) - f(x))^2 \, dy \, m^2(x) \, dx + C \int_{\R^d} f^2 \, m^2.
$$
We observe that 
$$
\begin{aligned}
&  -{1 \over 2}  \int_{\R^d \times \R^d} k_0(x-y) \, (f(y) - f(x))^2 \, dy \, m^2(x) \, dx \\
&\quad\le -{1 \over 4}  \int_{\R^d \times \R^d} k_0(x-y) \, ((fm)(y) - (fm)(x))^2 \, dy  \, dx \\
&\qquad + {1 \over 2} \int_{\R^d \times \R^d} k_0(x-y) \, (m(y) - m(x))^2 \, dx \, f^2(y) \, dy.
\end{aligned}
$$
We split the last term into two pieces, that we estimate in the following way: 
$$
\begin{aligned}
&\quad \int_{|x-y| \le 1} k_0(x-y) \, (m(y) - m(x))^2 \, dx \, f^2(y) \, dy \\
&\le  \int_0^1 \int_{|x-y| \le 1} k_0(x-y) \, |x-y|^2 \, |\nabla m(x+ \theta (y-x))|^2 \, dx \, f^2(y) \, dy d\theta \\
&\le C \, \int_{\R^d} f^2 \, m^2
\end{aligned}
$$
and 
$$
\begin{aligned}
& \int_{|x-y| \ge 1} k_0(x-y) \, (m(y) - m(x))^2 \, dx \, f^2(y) \, dy \\
&\quad\le C \, \int_{|x-y| \ge 1} k_0(x-y) \, (m^2(y) + m^2(y) \, m^2(x-y)) \, dx \, f^2(y) \, dy \\
&\quad\le C \, \int_{|z| \ge 1} k_0(z) \, m^p(z) \, dz \int_{\R^d} f^2 \, m^2  \le  C \, \int_{\R^d} f^2 \, m^2.
\end{aligned}
$$
We recall that the homogeneous Sobolev space $\dot{H}^s$ for $s\in\R$ is the set of tempered distributions $u$ such that $\widehat{u}$ belongs to $L^1_{loc}$ and 
$$
\|u\|^2_{\dot{H}^s} := \int_{\R^d} |\xi|^{2s} \, |\widehat{u}(\xi)|^2 \, d\xi < \infty, 
$$
and that for $s \in (0,1)$, there exists a constant $c_0>0$ such that 
$$
\|u\|^2_{\dot{H}^s} = c_0^{-1} \, \int_{\R^d \times \R^d} {{(u(x)-u(y))^2} \over {|x-y|^{d+2s}}} \, dx \, dy
$$
from which we deduce the following   identity:
\beqn \label{eq-DtoF:fractionalSobolev}
c_0 \, \|u\|^2_{\dot{H}^{\alpha/2}} = \int_{\R^d \times \R^d} (u(x)-u(y))^2 \, k_0(x-y) \, dx \, dy \quad \forall \, \alpha \in (0,2).
\eeqn 
As a consequence, up to change the value of $C$, we have proved
$$
T_1 \le -{c_0 \over 4}Â \| f \, m\|_{\dot{H}^{\alpha/2}}^2 + C\int_{\R^d} f^2 \, m^2.
$$
Next, we compute
$$
T_2 = \int_{\R^d} f^2 \, m^2 \, \left({d \over 2} - \frac{x \cdot \nabla m^2}{2 \, m^2} \right) \le {d \over 2} \int_{\R^d} f^2 \, m^2.
$$
Concerning $T_3$, we use Lemma \ref{lem-DtoF:Abounded} and Cauchy-Schwarz inequality:
$$
T_3 \le \|\AA_0 f\|_{L^2(m)} \, \|f\|_{L^2(m)} \le C \, \|f\|^2_{L^2(m)}.
$$
As a consequence, gathering the three previous inequalities, we have 
$$
\int_{\R^d}  (\BB_0 f) \, f \, m^2 \le -{c_0 \over 4}Â \| f \, m\|_{\dot{H}^{\alpha/2}}^2 + b_0 \int_{\R^d} f^2 \, m^2, \quad b_0 \in \R.
$$

\noindent {\it Step 2.} We now consider $b>b_0$ and we prove that for any $s \in \N$, $\BB_0 - b$ is hypodissipative in $H^s(m)$. For $s \in \N^*$, we introduce the  norm 
\beqn\label{eq-DtoC:Nt}
\Nt f \Nt^2_{H^s(m)} = \sum_{j=0}^s \eta^j \, \|Â \partial_x^j f \|^2_{L^2(m)}, \quad \eta>0,
\eeqn
which is equivalent to the classical $H^s(m)$ norm. 
We use again the fact that $\BB_0 = \Lambda_0 - \AA_0$ and we only deal with the case $s=1$, the higher order derivatives being treated in the same way. First, we have
$$
\partial_x(\BB_0 f) = \Lambda_0(\partial_x f) + \partial_x f - \partial_x (\AA_0 f).
$$
Then, we can notice that 
$$
\AA_0 f (x) = \int_{\R^d} k_0(z) \, \chi_{\eta,L}(z) \, \xi_R(x,x+z) \, f(x+z) \, dz
$$
so that 
$$
\partial_x (\AA_0 f)(x) = \AA_0 (\partial_x f) (x) + \widetilde{\AA_0} f(x), \quad \text{with} \quad \|\widetilde{\AA_0} f\|_{L^2(m)} \le C \, \|f\|_{L^2},
$$
where the last inequality is obtained thanks to inequality \eqref{eq-DtoF:ineqcaract} as in the proof of Lemma \ref{lem-DtoF:Abounded}. We deduce that 
$$
\partial_x(\BB_0 f) = \BB_0 (\partial_x f) + \partial_x f - \widetilde{\AA_0} f.
$$
Then, doing the same computations as in the case $s=0$, we obtain
$$
\begin{aligned}
&\quad \int_{\R^d} \partial_x(\BB_0 f) \, (\partial_x f) \, m^2 \\
&= \int_{\R^d} \BB_0 (\partial_x f)  \, (\partial_x f) \, m^2 + \int_{\R^d} (\partial_x f)^2 \, m^2 - \int_{\R^d} \widetilde{\AA_0} f \, (\partial_x f) \, m^2 \\
&=:J_1+J_2+J_3.
\end{aligned}
$$
with
$$
\begin{aligned}
J_1
&\le -{c_0 \over 4} \| (\partial_x f) m\|_{\dot{H}^{\alpha/2}}^2 + b_0 \int_{\R^d} (\partial_x f)^2 \, m^2 \\
&\le -{c_0 \over 8} \| f \, m\|^2_{\dot{H}^{1+\alpha/2}} + {c_0 \over 4} \|f \, \partial_x m\|^2_{\dot{H}^{\alpha/2}}  + b_0 \int_{\R^d} (\partial_x f)^2 \, m^2 \\
&\le -{c_0 \over 8} \| f \, m\|^2_{\dot{H}^{1+\alpha/2}} + C \left( \|f\|^2_{L^2(m)} + \|f \, m\|^2_{\dot{H}^1} \right), 
\end{aligned}
$$
and also
$$
J_2 \le {1 \over 2} \left(\|f\|^2_{L^2(m)} + \|f \, m\|^2_{\dot{H}^1} \right).
$$ 
Finally, using Cauchy-Schwarz inequality, we have
$$
J_3 \le \| \widetilde{\AA_0} f \|_{L^2(m)} \, \|\partial_x f \|_{L^2(m)} \le C \left(\|f\|^2_{L^2(m)} + \|f \, m\|^2_{\dot{H}^1} \right).
$$
As a consequence, we have
$$
\begin{aligned}
&\quad \int_{\R^d} \partial_x(\BB_0 f) \, (\partial_x f) \, m^2 \\
&\le -{c_0 \over 8} \| f \, m\|^2_{\dot{H}^{1+\alpha/2}} + b_1 \left(\|f\|^2_{L^2(m)} + \|f \, m\|^2_{\dot{H}^1} \right), \quad b_1 \in \R.
\end{aligned}
$$
We now introduce $f_t$ the solution to the evolution equation
$$
\partial_t f_t = \BB_0 f_t, \quad f_0=f, 
$$
and we compute
$$
\begin{aligned}
{1 \over 2} {d \over {dt}} \Nt f_t \Nt^2_{H^1(m)}
&= \int_{\R^d} (\BB_0 f_t) \, f_t \, m^2 + \eta \, \int_{\R^d}  \partial_x(\BB_0 f_t) \, (\partial_x f_t) \, m^2 \\
&\le -{c_0 \over 4}Â \| f_t \, m\|_{\dot{H}^{\alpha/2}}^2 -\eta {c_0 \over 8} \| f _t\, m\|^2_{\dot{H}^{1+\alpha/2}} \\
&\quad + \|f_t \|^2_{L^2(m)} (b_0 + \eta \,b_1) + \eta \, b_1\,\|f_t\, m\|^2_{\dot{H}^1}.
\end{aligned}
$$
We now use the following interpolation inequality
$$
\|h\|_{\dot{H}^1} \le \|h\|^{\alpha/2}_{\dot{H}^{\alpha/2}} \, \|h\|^{1-\alpha/2}_{\dot{H}^{1+\alpha/2}},
$$
which implies
\beqn \label{eq-DtoF:interp}
\|h\|^2_{\dot{H}^1} \le K(\zeta) \, \|h\|^2_{\dot{H}^{\alpha/2}} + \zeta \, \|h\|^2_{\dot{H}^{1+\alpha/2}}, \quad \zeta>0.
\eeqn
We obtain  
$$
\begin{aligned}
&\quad {1 \over 2} {d \over {dt}} \Nt f_t \Nt^2_{H^1(m)} \\
&\le \left(-{c_0 \over 4}Â + \eta \, b_1 \, K(\zeta)\right)\, \| f_t \, m\|_{\dot{H}^{\alpha/2}}^2 + \eta \, \left(-{c_0 \over 8} + \zeta  \, b_1\right) \, \| f _t\, m\|^2_{\dot{H}^{1+\alpha/2}} \\
&\quad + \|f_t \|^2_{L^2(m)} (b_0 + \eta \,b_1).
\end{aligned}
$$
Choosing $\zeta$ small enough so that $-c_0/8 + \zeta \, b_1 <0$ and then $\eta$ small enough so that $-c_0/4 + \eta \, b_1 \, K(\zeta)<0$ and $b_0 + \eta \,b_1<b$, we get
$$
{1 \over 2} {d \over {dt}} \Nt f_t \Nt^2_{H^1(m)} \le b \, \Nt f_t \Nt^2_{H^1(m)}
$$
which concludes the proof in the case $s=1$. 
\end{proof}

\smallskip
We now introduce the operator $\BB_{0,m}$ defined by 
\beqn \label{eq-DtoF:B0m}
\BB_{0,m}(h) = m \, \BB_0(m^{-1}h).
\eeqn
\begin{cor} \label{cor-DtoF:BdissipHs}
Consider $q$ such that $2q<\alpha$. There exists $b \in \R$ such that for any $s \in \N$, $\BB_{0,m} - b$ is hypodissipative in $H^s$.
\end{cor}
\begin{proof}
The proof comes from Lemma \ref{lem-DtoF:BdissipHs} and is immediate noticing that the norms defined on $H^s(m)$ by
$$
\|f\|_1^2 = \sum_{j=0}^s \|\partial^j_x f\|^2_{L^2(m)} \quad \text{and} \quad 
\| f \|_2^2 := \|f \, m\|_{H^s}^2
$$
are equivalent.
\end{proof}

\smallskip

\begin{lem} \label{lem-DtoF:BdissipH-s}
Consider $q$ such that $2q<\alpha$. There exists $b \in \R$ such that for any $s \in \N$, $\BB_{0,m} - b$ is hypodissipative in $H^{-s}$, (or equivalently, $\BB_0 - b$ is hypodissipative in $H^{-s}(m)$).
\end{lem}

\begin{proof}
We introduce the dual operator of $\BB_{0,m}$ defined by:
$$
\BB_{0,m}^* \phi = \omega \, I_0(m \, \phi) - x \cdot \nabla \phi - \frac{x \cdot \nabla m}{m} \, \phi - \omega \, \AA_0( m \, \phi)
$$
where $\omega := m^{-1}$.
We now want to prove that $\BB_{0,m}^*$ is hypodissipative in $H^s$. 

\noindent {\it Step 1.} We consider first the case $s=0$ and we compute
$$
\begin{aligned}
&\quad \int_{\R^d} (\BB_{0,m}^* \phi) \, \phi \\
&= \int_{\R^d} I_0(m \, \phi) \, \omega \, \phi  - \int_{\R^d} x \cdot (\nabla \phi) \, \phi - \int_{\R^d} \frac{x \cdot \nabla m}{m} \, \phi^2 - \int_{\R^d} \omega \, \AA_0(m \, \phi) \, \phi \\
&=: T_1 + \dots + T_4.
\end{aligned}
$$
We have 
$$
T_2 = {d \over 2} \int_{\R^d} \phi^2 \quad \text{and} \quad T_3 \le0.
$$
Next, using \eqref{eq-DtoF:ineqcaract}, we have $\|\AA_0 (m \, \phi) \|_{L^2} \le C \, \|\AA_0 (|\phi|)\|_{L^2}$ and thus
$$
T_4 \le C  \left(\|Â \AA_0 (|\phi|)\|^2 + \| \phi\|^2_{L^2} \right) \le C \, \|Â \phi\|^2_{L^2}
$$
from Lemma \ref{lem-DtoF:Abounded}. Let us now estimate $T_1$. 

\noindent {\it Case $\alpha<1$.} We write 
$$
\begin{aligned}
T_1 &= \int_{\R^d \times \R^d} k_0(x-y) \, ((m\phi)(y) - (m \phi)(x)) \, \omega(x) \, \phi(x) \, dy \, dx \\
&= \int_{\R^d \times \R^d} k_0(x-y) \, (\phi(y) - \phi(x)) \, \phi(x) \, dy \, dx \\
&\quad +\int_{|x-y|Â \le 1} k_0(x-y) \, (m(y) - m(x)) \, \omega(x) \, \phi(y) \, \phi(x) \, dy \, dx \\
&\quad +\int_{|x-y|Â \ge 1} k_0(x-y) \, (m(y) - m(x)) \, \omega(x) \, \phi(y) \, \phi(x) \, dy \, dx \\
&=: T_{11} + T_{12} + T_{13}.
\end{aligned}
$$
Let us point out here that from \eqref{eq-DtoF:fractionalSobolev}, we have
$$
\begin{aligned}
T_{11} &=  \int_{\R^d} I_0(\phi) \, \phi \\
&= -{1 \over 2}  \int_{\R^d \times \R^d} k_0(x-y) \, (\phi(y) - \phi(x))^2 \, dy \, dx + {1 \over 2} \int_{\R^d} I_0(\phi^2) \\
&= -{c_0 \over 2} \| \phi\|^2_{\dot{H}^{\alpha /2}}.
\end{aligned}
$$
Next, using a Taylor expansion, there exists $\theta \in (0,1)$ such that 
\beqn \label{eq-DtoF:T12}
\begin{aligned}
&T_{12} = \int_{|x-y|\le 1} k_0(x-y) \, (m(y) - m(x)) \, \omega(x) \, \phi(y) \, \phi(x) \, dy \, dx \\
&\le C \, \int_{|x-y|\le 1} k_0(x-y) \, |x-y| \, |\nabla m(x + \theta(y-x))| \, \omega(x) \, (\phi^2(y) + \phi^2(x)) \, dy \, dx.
\end{aligned}
\eeqn
Using that  $|\nabla m(x + \theta(y-x))| \, \omega(x) \le C$ for any $x$, $y \in \R^d$, $|x-y|\le1$, we deduce 
$$
T_{12} \le C \, \int_{\R^d} \phi^2.
$$
Concerning $T_{13}$, we have from \eqref{eq-DtoF:calculm}
$$
|m(y)-m(x)| \le  C \, \langle x-y \rangle^q \,\min\left(\langle x \rangle^{q/2}, \langle y \rangle^{q/2}\right),
$$
from which we deduce   
$$
T_{13}\le C \, \int_{\R^d} \phi^2.
$$
All together, we have thus proved 
$$
T_1 \le -{c_0 \over 2} \| \phi\|^2_{\dot{H}^{\alpha /2}} + C \, \int_{\R^d} \phi^2.
$$
\noindent {\it Case $\alpha \in [1,2)$.} We write
$$
\begin{aligned}
 &T_1 = 
 \int_{\R^d \times \R^d} k_0(x-y) \, \TT_{m\phi}(x,y) \, \omega(x) \, \phi(x) \, dy \, dx \\
&= \int_{\R^d} I_0(\phi) \, \phi + 
\int_{|x-y|\le 1} k_0(x-y) \, \TT_m(x,y) \, \omega(x) \, \phi(y) \, \phi(x) \, dy \, dx \\
&\quad + \int_{|x-y|\ge 1} k_0(x-y) \,  \TT_m(x,y) \,  \omega(x) \, \phi(y) \, \phi(x) \, dy \, dx \\
&\quad + \int_{\R^d \times \R^d} k_0(x-y) \, (\phi(y) - \phi(x)) \, \phi(x) \, \omega(x) \, \nabla m(x) \cdot (y-x) \, \chi(y-x) \, dy \, dx \\
&=:T_{11} + T_{12} + T_{13} + T_{14}
\end{aligned}
$$
where we recall that $\TT_\nu$ is defined in~\eqref{eq:psinu}. 
We have again
$$
T_{11} = -{c_0 \over 2} \|Â \phi\|^2_{\dot{H}^{\alpha/2}}.
$$
Arguing similarly as for $T_{12}$ in \eqref{eq-DtoF:T12}, but using a Taylor expansion at order~$2$ instead of order $1$, we obtain
$$
T_{12} \le C \, \int_{\R^d} \phi^2.
$$
Next, we split $T_{13}$ into two parts:
$$
\begin{aligned}
T_{13} &\le C \,\int_{|x-y|Â \ge 1} k_0(x-y) \, |m(y) - m(x)| \, \omega(x) (\phi^2(x) + \phi^2(y)) \, dx \, dy \\
&\quad +  C \,\int_{1 \le |x-y|Â \le 2} k_0(x-y) \, |x-y| \, |\nabla m(x)| \, \omega(x) \, (\phi^2(x) + \phi^2(y)) \, dx \, dy \\
&\le C \, \int_{|x-y|Â \ge 1} k_0(x-y) \, \langle x-y \rangle^q \, \langle x \rangle^{-q/2} \, (\phi^2(x) + \phi^2(y)) \, dx \, dy \\
&\quad +  C \int_{1 \le |x-y|Â \le 2} k_0(x-y) \,  (\phi^2(x) + \phi^2(y)) \, dx \, dy,
\end{aligned}
$$
where we have used \eqref{eq-DtoF:calculm}, we thus obtain:
$$
T_{13} \le C \, \int_{\R^d} \phi^2.
$$
Concerning $T_{14}$, we use Young inequality which implies that for any $\zeta>0$,
$$
\begin{aligned}
T_{14}
&\le \zeta \, \int_{\R^d \times \R^d} k_0(x-y) \, (\phi(y) - \phi(x))^2 \, dy \, dx \\
&\quad + K(\zeta) \, \int_{\R^d \times \R^d} k_0(x-y) \, \phi^2(x) \, {|\nabla m(x)|^2 \over m^2(x)} \, |y-x|^2 \, \chi^2(x-y) \, dy \, dx \\
&\le \zeta \, c_0 \, \|\phi\|^2_{\dot{H}^{\alpha/2}} + K(\zeta) \, \int_{|z| \le 2} k(z) \, |z|^2 \, dz \, \int_{\R^d} \phi^2.
\end{aligned}
$$
Consequently, taking $\zeta>0$ small enough, we have 
$$
T_1 \le -{c_0 \over 4} \|\phi\|^2_{\dot{H}^{\alpha/2}} + C \, \int_{\R^d} \phi^2.
$$
We hence conclude that 
$$
\int_{\R^d} (\BB_{0,m}^* \phi) \, \phi \le -{c_0 \over 4} \|\phi\|^2_{\dot{H}^{\alpha/2}} + b_0 \, \int_{\R^d} \phi^2, \quad b_0 \in \R.
$$

\noindent {\it Step 2.} We now consider $b>b_0$ and we prove that for any $s \in \N$, $\BB_{0,m}^* - b$ is hypodissipative in~$H^s$. As in \eqref{eq-DtoC:Nt}, for $s \in \N^*$, we introduce the norm 
$$
\Nt \phi \Nt^2_{H^s} := \sum_{j=0}^s \eta^j \, \|Â \partial^j_x \phi \|^2_{L^2}, \quad \eta>0, 
$$
which is equivalent to the classical $H^s$ norm. 
We only deal with the case $s=1$, the higher order derivatives are treated in the same way. First, using the identity \eqref{eq-DtoF:I(phim)} (with $k_0$ instead of $k_{0,\eps}$), we notice that 
$$
\BB_{0,m}^* \phi = I_0(\phi) + \omega \, \CC^1_m(\phi) +\omega \, \CC^2_m(\phi) - x \cdot \nabla \phi - \frac{x \cdot \nabla m}{m} \, \phi - \omega \, \AA_0( m \, \phi)
$$
where
$$
\begin{aligned}
\CC^1_m(\phi)(x) &= \int_{\R^d} k_0(x-y) \, \phi(y) \, (m(y)-m(x) - (y-x) \cdot \nabla m(x) \, \chi(x-y)) \, dy \\
&= \int_{\R^d} k_0(z) \, \phi(x+z) \, (m(x+z)-m(x) - z \cdot \nabla m(x) \, \chi(z)) \, dz 
\end{aligned}
$$
and
$$
\begin{aligned}
\CC^2_m(\phi)(x) &= \int_{\R^d} k_0(x-y) \, (\phi(y) - \phi(x)) \, \nabla m(x) \cdot (y-x) \, \chi(x-y) \, dy \\
&= \int_{\R^d} k_0(z) \, (\phi(x+z) - \phi(x)) \, \nabla m(x) \cdot z \, \chi(z) \, dz.
\end{aligned}
$$
Before going into the computation of $\partial_x (\BB_{0,m}^* \phi)$, we also notice that 
$$
\partial_x \left(\omega \, \AA_0( m \, \phi)\right) = \omega \, \AA_0(m \, \partial_x \phi) + \widehat{\AA_{0,m}}(\phi)
$$
where $\widehat{\AA_{0,m}}$ satisfies 
$$
\| \widehat{\AA_{0,m}}(\phi) \|_{L^2} \le C \, \|\phi\|_{L^2}
$$
thanks to \eqref{eq-DtoF:ineqcaract}.
Consequently, we have 
$$
\begin{aligned}
\partial_x (\BB_{0,m}^* \phi) &= \BB_{0,m}^* (\partial_x \phi) + \omega \, \CC^1_{\partial_x m} (\phi) + \omega \, \CC^2_{\partial_x m} (\phi) + \partial_x \omega \,\CC^1_m(\phi)  +\partial_x \omega \, \CC^2_m(\phi) \\
&\quad - \partial_x \phi - \partial_x \left( {{x \cdot \nabla m} \over m} \right) \phi - \widehat{\AA_{0,m}} (\phi)
\end{aligned}
$$
and
$$
\begin{aligned}
&\quad \int_{\R^d} \partial_x (\BB_{0,m}^* \phi) \, \partial_x \phi \\
&= \int_{\R^d} \BB_{0,m}^* (\partial_x \phi) \, (\partial_x \phi)
+ \int_{\R^d} \omega \, \CC^1_{\partial_x m} (\phi) \, (\partial_x \phi)
+ \int_{\R^d} \omega \, \CC^2_{\partial_x m} (\phi) \, (\partial_x \phi) \\
&\quad + \int_{\R^d}  \partial_x \omega \,\CC^1_m(\phi) \, (\partial_x \phi) 
+  \int_{\R^d}  \partial_x \omega \,\CC^2_m(\phi) \, (\partial_x \phi)  - \int_{\R^d} (\partial_x \phi)^2 \\
&\quad -  \int_{\R^d} \partial_x \left( {{x \cdot \nabla m} \over m} \right) \, \phi \, (\partial_x \phi)
-  \int_{\R^d} \widehat{\AA_{0,m}} (\phi) \, (\partial_x \phi) \\
&=: J_1 + \dots + J_8.
\end{aligned}
$$
We have from the step 1 of the proof
$$
J_1 \le -{c_0 \over 4} \| \phi\|^2_{\dot{H}^{1 + \alpha/2}} + b_0 \int_{\R^d} (\partial_x \phi)^2.
$$
Moreover, we easily obtain that 
$$
J_6 + J_7+ J_8 \le C \, \left(\int_{\R^d} \phi^2 + \int_{\R^d} (\partial_x \phi)^2 \right).
$$
The term $J_2$ is first separated into two parts:
$$
\begin{aligned}
J_2 &=
 \int_{|z| \le 1} k_0(z) \, \phi(y) \, \TT_{\partial_xm}(x,x+z) \, \omega(x) \, \partial_x \phi(x) \,  dz \, dx  \\
& \quad + \int_{|z| \ge 1} k_0(z)  \phi(y)  \TT_{\partial_xm}(x,x+z)  \, \omega(x)  \partial_x \phi(x) \,  dz \, dx \\
&=: J_{21} + J_{22}
\end{aligned}
$$
where we recall that $\TT_{\partial_x m}$ is defined in~\eqref{eq:psinu}. 
The term $J_{21}$ is treated as $T_{12}$ is the step $1$ of the proof. Concerning $J_{22}$, as for $T_{13}$, we split it into two parts: 
$$
\begin{aligned}
&\quad J_{22} \le \\
& \int_{|z|Â \ge 1} k_0(z) \, |(\partial_x m)(x+z) - (\partial_x m)(x)| \, \omega(x) (\phi^2(x+z) + (\partial_x\phi)^2(x)) \, dx \, dz \\
&\quad +  \int_{1 \le |z|Â \le 2} k_0(z) \, |z| \, |\nabla (\partial_x m)(x)| \, \omega(x) \, (\phi^2(x+z) + (\partial_x\phi)^2(x+z)) \, dx \, dz \\
&\le C \,  \int_{|z|Â \ge 1} k_0(z) \,(\phi^2(x+z) + (\partial_x\phi)^2(x)) \, dx \, dz \\
&\quad+C \,  \int_{1 \le |z|Â \le 2} k_0(z) \, (\phi^2(x+z) + (\partial_x\phi)^2(x+z)) \, dx \, dz,
\end{aligned}
$$
where the second inequality comes from the fact that 
$$
|(\partial_x m)(y) - (\partial_x m)(x)| \, \omega(x) \le C \quad \text{and} \quad  |\nabla (\partial_x m)(x)| \, \omega(x) \le C \quad \forall \, x, \, y \in \R^d
$$
because $q<\alpha/2<1$. 
We hence deduce that 
$$
J_2 \le C \, \left( \int_{\R^d} \phi^2 + \int_{\R^d} (\partial_x \phi)^2 \right).
$$
Concerning $J_3$, we perform a Taylor expansion of $\phi$ and use the fact that $|\nabla (\partial_x m)| \, \omega \in L^\infty(\R^d)$:
\beqn \label{eq-DtoF:J3}
\begin{aligned}
J_3 &= \int_{\R^d \times \R^d} k_0(x-y) \,\int_0^1 (1-t)\,  \nabla \phi(y+t(x-y)) \cdot (y-x) \, dt \, \\
&\qquad \qquad \qquad \qquad \nabla (\partial_x m)(x) \cdot (y-x) \, \chi(x-y) \, \omega(x) \, \partial_x\phi(x) \, dy \, dx \\
&\le C \, \int_{|z| \le 2} k_0(z) \, |z|^2 \, \int_0^1|\nabla \phi(x+tz)|^2 \, dt \, dz \, dx \\
&\quad + \int_{|z| \le 2} k_0(z) \, |z|^2 \, |\partial_x \phi (x)|^2 \, dz \, dx,
\end{aligned}
\eeqn
where we have used Jensen inequality 
and Young inequality. We use a change of variable for the first term of the RHS of \eqref{eq-DtoF:J3}, which implies that 
$$
J_3 \le C \, \|\phi\|^2_{\dot{H}^1}.
$$
We deal with $J_4$ splitting it into two parts ($|x-y|\le 1$ and $|x-y| \ge1$) and using the same method as for $T_{12}$ and $T_{13}$ in the step 1 of the proof, we obtain
$$
J_4 \le C \, \left(\int_{\R^d} \phi^2 + \int_{\R^d} (\partial_x \phi)^2 \right).
$$
To deal with $J_5$, we proceed exactly as for $J_3$ and we obtain 
$$
J_5 \le C \, \|\phi\|^2_{\dot{H}^1}.
$$

Summarizing the previous inequalities and using \eqref{eq-DtoF:interp}, we obtain that for any $\zeta >0$,
$$
\begin{aligned}
&\quad \int_{\R^d} \partial_x (\BB_{0,m}^* \phi) \, \partial_x \phi \le -{c_0 \over 4} \| \phi\|^2_{\dot{H}^{1 + \alpha/2}} + b_1 \left(\|\phi\|^2_{L^2} + \|Â \phi\|^2_{\dot{H}^1} \right)\\
&\le -{c_0 \over 4} \| \phi\|^2_{\dot{H}^{1 + \alpha/2}} + b_1  \left(\|\phi\|^2_{L^2} + K(\zeta) \| \phi\|^2_{\dot{H}^{\alpha/2}} + \zeta \|\phi\|^2_{\dot{H}^{1 + \alpha/2}} \right), \quad b_1 \in \R. 
\end{aligned}
$$
This implies that if $\phi_t$ is the solution of 
$$
\partial_t \phi_t = \BB^*_{0,m} \phi_t, \quad \phi_0=\phi
$$
then
$$
\begin{aligned}
 {1 \over 2} {d \over dt} \Nt \phi_t \Nt_{H^1}^2 
\le &\left(-{c_0 \over 4} + \eta \, b_1 \, K(\zeta) \right) \|\phi_t\|^2_{\dot{H}^{\alpha/2}} \\
&\quad + \eta \left(-{c_0 \over 4} + \zeta \, b_1 \right) \|\phi_t\|^2_{\dot{H}^{1+\alpha/2}} + (b_0 + \eta \, b_1) \|Â \phi_t\|^2_{L^2}.
\end{aligned}
$$
Taking $\zeta$ and $\eta$ small enough, we deduce that 
$$
{1 \over 2} {d \over dt} \Nt \phi_t \Nt^2_{H^1} \le b \, \Nt \phi_t \Nt^2_{H^1},
$$
this concludes the proof in the case $s=1$. 
\end{proof}

We now fix $0<r<\alpha/2$ as in the assumptions of Theorem~\ref{theo-DtoF:DtoF}. We also introduce $r_0 \in (r,\alpha/2)$ and $m_0(x) := \langle x \rangle^{r_0}$.
From Lemma \ref{lem-DtoF:BdissipLp} applied with $p=1$, there exists $a<0$ such that $\BB_\eps - a$ is dissipative in $L^1(m_0)$ for any $\eps \in [0,\eps_1]$ (or equivalently, $\BB_{\eps,m_0} -a$ is dissipative in $L^1$ where $\BB_{\eps,m_0}$ is defined as $\BB_{0,m}$ in \eqref{eq-DtoF:B0m}). From Lemma~\ref{lem-DtoF:BdissipLp} applied with $p=2$, Corollary \ref{cor-DtoF:BdissipHs} and Lemma~\ref{lem-DtoF:BdissipH-s}, there exists $b \in \R$ such that $\BB_\eps -b$ is dissipative in $L^2(m_0)$ for any $\eps \in [0,\eps_1]$ (or equivalently, $\BB_{\eps,m_0} -b$ is dissipative in $L^2$)
and $\BB_{0,m_0} -b$ is hypodissipative in $H^s$ and $H^{-s}$ for any $s \in \N^*$.

We introduce $p_\theta :=  2/(1+\theta)$ and its H\"{o}lder conjugate $p'_\theta := 2/(1-\theta)$ for $\theta \in (0,1)$. We then choose $\theta \in (0,1)$ such that $a_\theta := a \theta + b(1-\theta) <0$, $p'_\theta \in \N$ and $p'_\theta(r_0-r)>d$. We denote 
$$
\begin{aligned}
X_1:=W^{2,p_\theta}(m_0) \subset X_0:=L^{p_\theta}(m_0)\subset X_{-1} := W^{-2,p_\theta}(m_0).
\end{aligned}
$$

\begin{lem} \label{lem-DtoF:BdissipXi}
The operator $\BB_{0} - a_\theta$ is hypodissipative in $X_i$, $i=-1,0,1$ and the operator $\BB_\eps - a_\theta$ is dissipative in $X_0$ for any $\eps \in (0,\eps_1]$.
\end{lem}
\begin{proof}
We prove that $\BB_{0,m_0} - a_\theta$ is hypodissipative in $W^{-2,p_\theta}$, $L^{p_\theta}$ and $W^{2,p_\theta}$ by interpolation. To conclude for $X_0$, we just have to interpolate the results coming from Lemma \ref{lem-DtoF:BdissipLp} with $p=1$ and Lemma \ref{lem-DtoF:BdissipHs} with $s=0$ and use the fact that $\left[L^1,L^2\right]_\theta = L^{p_\theta}$ with 
$1/p_\theta = \theta + (1-\theta)/2$ i.e. $p_\theta = 2/(1+\theta)$. Then, for $X_1$ and $X_{-1}$, we first choose $s_0$ large enough so that $s_0(1-\theta)=2$. We then have $\left[L^1,H^{s_0}\right]_\theta = W^{2,p_\theta}$, $\left[L^1,H^{-s_0}\right]_\theta = W^{-2,p_\theta}$ and we conclude thanks to Lemma \ref{lem-DtoF:BdissipLp} with $p=1$ and Lemma \ref{lem-DtoF:BdissipHs} with $s=s_0$. 

We prove that $\BB_\eps - a_\theta$ is dissipative in $X_0$ exactly in the same way as we proved that $\BB_0 - a_\theta$ is dissipative in $X_0$.
\end{proof}

\smallskip
\subsection{Spectral analysis}
We here divide the proof of Theorem \ref{theo-DtoF:DtoF} into two parts, using Krein Rutman theory for the first part and using both perturbative and enlargement arguments for the second part. 

\smallskip
\noindent {\it Proof of part (1) of Theorem \ref{theo-DtoF:DtoF}.} First, we notice that as in Section \ref{sec:DFP-FP} (Lemmas \ref{lem-DtoC:Kato} and \ref{lem-DtoC:strongPM}), we can prove that the operator $\Lambda_\eps$ satisfies Kato's inequalities, $S_{\Lambda_\eps}$ is a positive semigroup and $(-\Lambda_\eps)$ satisfies a strong maximum principle. Using Krein-Rutman theory, this gives the first part of Theorem~\ref{theo-DtoF:DtoF} i.e. that there exists a unique $G_\eps>0$ such that $\|G_\eps\|_{L^1}=1$, $\Lambda_\eps G_\eps =0$. Moreover, it also implies that $\Pi_\eps f = \langle f \rangle G_\eps$.

\smallskip 
\noindent {\it Proof of part (2) of Theorem \ref{theo-DtoF:DtoF}.} We first develop a perturbative argument which is detailed in what follows, improving a bit similar results presented in~\cite{MM*,Granular-IT*}. We then ends the proof using an enlargement argument. 
\begin{lem} For any $z \in \Omega :=  D_{a_\theta} \backslash\{0\}$ we define the family of operators 
$$
K_{\eps} (z):=  - (\Lambda_{\eps}-\Lambda_0) \, \RR_{\Lambda_0}(z) \, (\AA  \RR_{\BB_\eps} (z)). 
$$
There exists a function $\eta_{2}(\eps) \xrightarrow[\eps \rightarrow 0]{}0$ such that   
\beqn \label{eq-DtoF:Keps}
\| K_{\eps} (z) \|_{\BBB(X_{0})} \le \eta_{2}(\eps) \quad
\forall \, z \in  \Omega_\eps := \Delta_a \backslash \bar B_\eps, \quad  B_{\eps} := B(0,\eta_2(\eps)).
\eeqn
Moreover, there exists $\eps_2 \in(0,\eps_1)$ such that for any $\eps \in (0,\eps_2)$ the operators $I+K_{\eps} (z)$ and  $\Lambda_\eps-z$ 
are invertible for any $z \in \Omega_\eps$ and
$$
\forall \, z \in \Omega_\eps, \quad \RR_{\Lambda_{\eps}}(z) = \UU_\eps (z) \,  (I + K_{\eps}(z))^{-1}
$$
with
$$
\UU_\eps   := \RR_{\BB_\eps }  - \RR_{\Lambda_0} (\AA \, \RR_{\BB_\eps}).
$$
As an immediate consequence, there holds 
$$
\Sigma(\Lambda_{\eps}) \cap  D_{a_\theta} \subset \bar B_\eps .
$$
\end{lem}

\begin{proof} We know that the operators  $\AA \RR_{\BB_\eps}(z) : X_0 \to X_1$ (from Lemmas \ref{lem-DtoF:Abounded} and \ref{lem-DtoF:BdissipXi}) and $\RR_{\Lambda_0}(z) : X_1 \to X_1$ (previous works from \cite{GMM,MM*}) are bounded for any $z \in \Omega$ and that the operators  $\Lambda_{\eps}-\Lambda_0 : X_{1} \to X_{0}$ are small as $\eps \to 0$ uniformly in $z \in \Omega$ (Lemma \ref{lem-DtoF:convergence}). Because $0$ is a simple eigenvalue, we have 
$$
\| \RR_{\Lambda_0}(z)  \|_{\BBB(X_1)} \le C \, |z|^{-1} \quad \forall \, z \in \Omega.
$$
for some $C > 0$. We introduce the constant $C_{a_\theta}>0$ (coming from Lemmas~\ref{lem-DtoF:Abounded} and \ref{lem-DtoF:BdissipXi}) such that 
$$
\|\AA S_{\BB_\eps}(t)\|_{\BBB(X_0,X_1)} \le C_{a_\theta} \, e^{a_\theta t}.
$$
Defining $\eta_{2}(\eps) := (C \, C_{a_\theta} \, \eta_{1}(\eps))^{1/2} $, we deduce that for any $z \in \Omega_{\eps}$,
\beqn\label{eq-DtoF:Kepseta}
\| K_{\eps} (z) \|_{\BBB(X_{0})} \le \eta_{1}(\eps) \, { C \over \eta_{2}(\eps) } \, C_{a_\theta} = \eta_{2}(\eps). 
\eeqn
We choose $\eps_2>0$ such that $\eta_2(\eps) < 1$ for any $\eps \in (0,\eps_2)$, we thus obtain that $\| K_{\eps}(z) \| < 1$ for any $\eps\in(0,\eps_2)$ and $z \in \Omega_\eps$, which implies that  $I + K_{\eps}(z)$ is invertible.

\smallskip
We compute
\bean
 (\Lambda_{\eps} - z)  \, \UU_\eps 
&=&  (\BB_\eps - z + \AA) \RR_{\BB_\eps }    
-  (\Lambda_{\eps} - \Lambda_0 + \Lambda_0 - z) \RR_{\Lambda_0} \,  \AA \, \RR_{\BB_\eps}   
\\
&=& Id   +  K_{\eps}.
 \eean
For $z \in \Omega_\eps$, $\eps \in (0,\eps_2)$, we denote $ \JJ_{\eps}(z):= \UU_{\eps}(z) \, (I + K_{\eps}(z))^{-1}$, so that  
 $$
 (\Lambda_{\eps} - z) \,  \JJ_{\eps}(z) = Id,
 $$
 which implies that $\Lambda_{\eps} - z$ has a right-inverse  $\JJ_{\eps}(z)$.
 
 \smallskip
 Since $\Lambda_{\eps} - z$ is invertible for $\Re e \, z $ large enough and $\JJ_{\eps}(z)$ is uniformly locally bounded in $\Omega_{\eps}$, we deduce that $\Lambda_{\eps} - z$ is invertible in $\Omega_{\eps}$, and its inverse is its right-inverse $\JJ_{\eps}(z)$. 
\end{proof}

\begin{lem} \label{lem-DtoF:proj}
 Let us denote 
 $$
 \Pi_\eps := {i \over 2\pi} \int_{\Gamma_{\eps}} \RR_{\Lambda_\eps}(z) \, dz, \quad \Gamma_\eps := \{z \in \C: |z| = \eta_2(\eps) \}
 $$
 the spectral projector onto eigenspaces associated to eigenvalues contained in~$\bar B_\eps$. There exists $\eta_3(\eps)$ such that 
 $$
 \|Â \Pi_\eps - \Pi_0 \|_{\BBB(X_0)} \le \eta_3(\eps) \xrightarrow[\eps \rightarrow 0]{}0.
 $$
\end{lem}

\begin{proof}
First, we have 
$$
\begin{aligned}
\Pi_\eps &={i \over {2\pi}} \int_{\Gamma_\eps} \left\{\RR_{\BB_\eps}(z) - \RR_{\Lambda_0}(z) (\AA \RR_{\BB_\eps}(z)) \right\} (I+K_\eps(z))^{-1} \, dz \\
&= {i \over {2\pi}} \int_{\Gamma_\eps} \RR_{\BB_\eps}(z) \left\{I - K_\eps(z) (I+K_\eps(z))^{-1}\right\} \, dz\\
&\quad - {i \over {2\pi}} \int_{\Gamma_\eps} \RR_{\Lambda_0}(z) (\AA \RR_{\BB_\eps}(z)) \left\{I - K_\eps(z) (I+K_\eps(z))^{-1}\right\} \, dz \\
&= {1 \over {2i\pi}} \int_{\Gamma_\eps}  \RR_{\BB_\eps}(z) K_\eps(z) (I+K_\eps(z))^{-1}\, dz \\
&\quad - {i \over {2\pi}} \int_{\Gamma_\eps} \RR_{\Lambda_0}(z) (\AA \RR_{\BB_\eps}(z)) \left\{I - K_\eps(z) (I+K_\eps(z))^{-1}\right\} \, dz \\
\end{aligned}
$$
and similarly,
$$
\begin{aligned}
\Pi_0 &= {i \over {2\pi}} \int_{\Gamma_\eps} \RR_{\Lambda_0}(z) \, dz \\
&= {i \over {2\pi}} \int_{\Gamma_\eps} \left\{\RR_{\BB_0}(z) - \RR_{\Lambda_0}(z) \, (\AA \RR_{\BB_0}(z)) \right\}\, dz \\
&= {1 \over {2i\pi}} \int_{\Gamma_\eps} \RR_{\Lambda_0}(z) \, (\AA \RR_{\BB_0}(z)) \, dz.
\end{aligned}
$$
Consequently, 
$$
\begin{aligned}
\Pi_0 - \Pi_\eps &= {1 \over {2i\pi}} \int_{\Gamma_\eps} \RR_{\Lambda_0}(z) \, \left\{ \AA \RR_{\BB_0}(z) - \AA \RR_{\BB_\eps}(z) \right\} \, dz \\
&\quad - {1 \over {2i\pi}} \int_{\Gamma_\eps} 
\left\{ \RR_{\BB_\eps}(z) - \RR_{\Lambda_0}(z) \AA \RR_{\BB_\eps}(z) \right\} K_\eps(z) (I+K_\eps(z))^{-1} \, dz \\
&=: T_1 + T_2.
\end{aligned}
$$
Concerning $T_1$, we use the identity 
$$
\AA \RR_{\BB_0}(z) - \AA \RR_{\BB_\eps}(z) = \AA \RR_{\BB_0}(z) (\BB_\eps - \BB_0) \RR_{\BB_\eps}(z)
$$
with Lemmas \ref{lem-DtoF:convergence}, \ref{lem-DtoF:Abounded} and \ref{lem-DtoF:BdissipXi} which imply that 
$$
\RR_{\BB_\eps}(z) \in \BBB(X_0), \, \, \,  \|\BB_\eps - \BB_0\|_{X_0 \to X_{-1}} \le \eta_1(\eps) \xrightarrow[\eps \rightarrow 0]{}0, \, \, \,  \AA \RR_{\BB_0}(z) \in \BBB(X_{-1},X_0).
$$
To treat $T_2$, we use estimate \eqref{eq-DtoF:Keps} on $K_\eps(z)$, the facts that $\RR_{\BB_\eps}(z) \in \BBB(X_0)$ and that we also have $\RR_{\Lambda_0}(z) \AA \RR_{\BB_\eps}(z) \in \BBB(X_0)$. That concludes the proof. 
\end{proof}

 \begin{prop} \label{prop-DtoF:spectrLeps}There exists $\eps_{0} \in (0,\eps_2)$ such that for any $\eps \in (0,\eps_0)$, the following properties hold in $X_0$:
 \begin{enumerate}
 \item $
 \Sigma(\Lambda_\eps) \cap  D_{a_\theta} = \{0\};
 $
 \item for any $f \in X_0$ and any $a>a_\theta$, 
$$
\| S_{\Lambda_\eps}(t) f - G_\eps  \langle f \rangle  \|_{X_0} \leq C_a \, e^{at} \, \| f- G_\eps \langle f \rangle\|_{X_0}, \quad \forall \, t \ge0
$$
for some explicit constant $C_a\ge1$. 
\end{enumerate}
 \end{prop}

\begin{proof} We know that if $P$ and $Q$ are two projectors s.t. $\|P-Q\|_{\BBB(X_0)} <1$, then their ranges are isomorphic. Lemma \ref{lem-DtoF:proj} thus implies that there exists $\eps_0 \in (0,\eps_1)$ such that for any $\eps \in (0,\eps_0)$,
 $$
 \hbox{\rm dim}\, \mbox{\rm R} (\Pi_\eps)  =  \hbox{\rm dim}\, \mbox{\rm R} (\Pi_0) =1.
 $$
We also know that $0$ is an eigenvalue for $\Lambda_\eps$ (cf. part (1) of Theorem \ref{theo-DtoF:DtoF}). This concludes the proof of the first part of the proposition. 

To get the estimate on the semigroup, we use a spectral mapping  theorem coming from \cite[Theorem~2.1]{MS}. The hypothesis of the theorem are satisfied because $\BB_\eps - a$ is hypodissipative in $X_0$ (and thus in $D(\Lambda_{\eps_{|X_0}}) = D(\BB_{\eps_{|X_0}})$) and $\AA \in \BBB(X_0, W^{2,p_\theta}_1(m))$ (and thus $\AA \in \BBB(X_0,D(\Lambda_{\eps_{|X_0}}))$.
\end{proof}

To conclude the proof of part (2) of Theorem \ref{theo-DtoF:DtoF}, we use the previous Proposition~\ref{prop-DtoF:spectrLeps} combined with an enlargement argument (see \cite{GMM} or \cite[Theorem 1.1]{MM*}): our ``small space'' is $E=L^{p_\theta}_{r_0}$ and our ``large'' space is $\EE=L^1_r$. We then use Lemmas~\ref{lem-DtoF:Abounded} and \ref{lem-DtoF:BdissipLp}-\ref{lem-DtoF:BdissipXi}, and the fact that we clearly have $\AA \in \BBB(\EE,E)$.

\bigskip

\bibliographystyle{acm}

\end{document}